\documentclass[12pt,twoside,reqno]{amsart}

\usepackage[margin=1.4in]{geometry}
\usepackage[hidelinks]{hyperref}
\renewcommand{\Bbb}{\mathbb}

\usepackage{color}
\usepackage{graphicx}
\usepackage{amsmath}
\usepackage{amssymb}
\usepackage{amsfonts}
\usepackage{multicol}
\usepackage[ansinew]{inputenc}
\usepackage{amscd}
\usepackage[mathscr]{eucal}
\usepackage{hyperref}

\newcommand{\R}{\mathbb{R}}
\newcommand{\C}{\mathbb{C}}
\newcommand{\N}{\mathbb{N}}
\newcommand{\Z}{\mathbb{Z}}
\newcommand{\Q}{\mathbb{Q}}
\newcommand{\T}{\mathbb{T}}

\newcommand{\Su}{\mathbb{S}}
\newcommand{\SL}{{\rm SL}}

\newcommand{\GL}{{\rm GL}^+}

\newcommand{\Mat}{{\rm Mat}}

\newcommand{\tr}{\mbox{tr}}

\newcommand{\conorm}{\mathrm{m}}

\newcommand{\proj}[1]{\hat{#1}}
\newcommand{\abs}[1]{\bigl| #1 \bigr|} % absolute value
\newcommand{\norm}[1]{\lVert#1\rVert} % norm
\newcommand{\normtwo}[1]{% Peter Grill norm @tex.stackexchange.com
	{\left\vert\kern-0.25ex\left\vert\kern-0.25ex\left\vert #1
		\right\vert\kern-0.25ex\right\vert\kern-0.25ex\right\vert} }

\newcommand{\length}{\rm length}

\newcommand{\Pp}{\mathbb{P}}

\newcommand{\U}{\mathscr{U}}

\newcommand{\medir}{{\overline{v}}}
\newcommand{\ledir}{{\underline{v}}}

\newcommand{\Bscr}{\mathscr{B}}

\newcommand{\vb}{\mathbf{v}}

\newcommand{\bld}[1]{\mathbf{{#1}}}

\newcommand\restr[2]{{% we make the whole thing an ordinary symbol
		\left.\kern-\nulldelimiterspace % automatically resize the bar with \right
		#1 % the function
		\vphantom{\big|} % pretend it's a little taller at normal size
		\right|_{#2} % this is the delimiter
}}
\theoremstyle{plain}
\newtheorem{theorem}{Theorem}[section]
\newtheorem{proposition}{Proposition}[section]
\newtheorem{corollary}[proposition]{Corollary}
\newtheorem{lemma}[proposition]{Lemma}

\theoremstyle{definition}
\newtheorem{definition}{Definition}[section]

\newtheorem{assumption}{Assumption}

\theoremstyle{definition}
\newtheorem{remark}{Remark}[section]
\newtheorem{example}[theorem]{Example}
\numberwithin{equation}{section}

\newcommand{\Prob}{\mathrm{Prob}}

\newcommand{\Imp}{\mathrm{Im}\,}

\newcommand{\rank}{\mathrm{rank}}
\newcommand{\Ker}{\mathrm{Ker}}
\newcommand{\supp}{\mathrm{supp}}

\title {A dynamical Thouless formula}

\date{}

\begin{document}

\author[J. Bezerra]{Jamerson Bezerra}
\address{Faculty of Mathematics and Computer Science, Nicolaus Co\-pernicus University (UMK), Poland.}
\email{jdouglas@impa.br}

\author[A. Cai]{Ao Cai}
\address{Departamento de Matem\'atica, Pontif\'icia Universidade Cat\'olica do Rio de Janeiro (PUC-Rio), Brazil}
\email{godcaiao@gmail.com}

\author[P. Duarte]{Pedro Duarte}
\address{Departamento de Matem\'atica and CMAFcIO\\
Faculdade de Ci\^encias\\
Universidade de Lisboa\\
Portugal
}
\email{pmduarte@fc.ul.pt}

\author[C. Freijo]{Catalina Freijo}
\address{Instituto de Matem\'{a}tica e Estat\'{i}stica, Universidade de S\~{a}o Paulo, Brazil}
\email{catalinafreijo@gmail.com}

\author[S. Klein]{Silvius Klein}
\address{Departamento de Matem\'atica, Pontif\'icia Universidade Cat\'olica do Rio de Janeiro (PUC-Rio), Brazil}
\email{silviusk@mat.puc-rio.br}

\begin{abstract}
In this paper we establish an abstract, dynamical Thouless-type formula for affine families of $\rm{GL} (2,\R)$ cocycles. This result extends the classical formula relating, via the Hilbert transform, the maximal Lyapunov exponent and the integrated density of states of  a Schr\"odinger operator. Here, the role of the integrated density of states will be played by a more geometrical quantity, the fibered rotation number. As an application of this formula we present limitations on the modulus of continuity of random linear cocycles. Moreover, we derive H\"older-type continuity properties of the fibered rotation number for linear cocycles over various base dynamics.
\end{abstract}

\maketitle

\setcounter{tocdepth}{1}
\tableofcontents

\section{Introduction and statements}\label{intro}
Thouless formula relates a mathematical physics object, the integrated density of states (IDS) of a Schr\"odinger operator, and a dynamical systems object, the Lyapunov exponent (LE) of the operator, via a singular integral operator, namely the Hilbert transform. It is named after the British condensed-matter physicist David J. Thouless, who formulated it in the context of the one dimensional Anderson model and proved it (not completely rigorously) in~\cite{Thou72}. The result was later extended and proven rigorously by Avron and Simon~\cite{AvronSimon}, Craig and Simon~\cite{CS83} and others. Let us describe it more precisely.

Consider an invertible ergodic transformation  $T\colon X \to X$ 
over a probability space $(X,\mu)$. Given a bounded and  measurable observable $\upsilon \colon X \to \R$,
let $v_n (x) := \upsilon (T^n x)$ for all $x \in X$ and $n \in \Z$. 

Denote by $l^2 (\Z)$ the Hilbert space of square summable sequences of real numbers $(\psi_n)_{n\in\Z}$.  
The {\em discrete ergodic Schr\"odinger operator} with potential
$n\mapsto \upsilon_n(x)$  is the operator $H (x)$ defined on $l^2 (\Z) \ni \psi  = \{\psi_n\}_{n \in \Z}$ by
\begin{equation}\label{ s op}
[ H (x) \, \psi ]_n := - (\psi_{n+1} + \psi_{n-1}) +  v_n (x) \, \psi_n\, .
\end{equation} 

Note that due to the ergodicity of the system, the spectral properties of the family of operators $\{ H (x) \colon x \in X \}$ are $\mu$-a.s. independent of the phase $x$. 

Given an energy parameter $E\in\R$, the Schr\"odinger (or eigenvalue) equation $H (x) \, \psi = E \, \psi$ can be solved formally by means of the iterates of a certain dynamical system. More precisely, consider the  associated  {\em Schr\"odinger cocycle} 
$ X \times \R^2 \ni (x, v) \mapsto (T x, A_E (x) \, v) \in X \times \R^2$, where  
$A_E \colon X \to \SL_2 (\R)$ is given by
$$A_{  E} (x) := \left[ \begin{array}{ccc} 
 \upsilon (x)  - E  & &  -1  \\
1 & &  \phantom{-}0 \\  \end{array} \right] = \left[ \begin{array}{ccc} 
 \upsilon (x)    & &  -1  \\
1 & &  \phantom{-}0 \\  \end{array} \right] + E \, \left[ \begin{array}{ccc} 
 -1   & &  0  \\
\phantom{-}0 & &  0 \\  \end{array} \right] \,.$$

Let $A^n_{ E}$ denote the $n$-th iterate of the cocycle, that is,
$$A^n_E (x) = A_E (T^{n-1} x) \cdots A_E (Tx) \, A_E (x) \, .$$ 

Then the formal solution of the Schr\"odinger equation $H (x) \, \psi = E \, \psi$ is given by
\begin{align} 
\label{schrodinger formal solution}
\left[\begin{array}{c}
\psi_{n}\\
\psi_{n-1} \\ \end{array}\right]  =  
A^n_{ E} (x)
   \left[\begin{array}{c}
\psi_0\\
\psi_{-1} \\ \end{array}\right]\, .
\end{align}

The average asymptotic growth rate of the iterates of the Schr\"odinger cocycle $A_E$ is called the maximal {\em Lyapunov exponent}, denoted by $L_1 (A_E)$. Moreover, the {\em integrated density of states} $N(E)$ measures, in some sense, how many states correspond to energies below the level $E$. Thouless formula establishes the following relation  between these two quantities:
$$
L_1 (A_E)=\int \log\abs{E-E'}dN(E') \, , 
$$
where the integral above is in the Lebesgue-Stieljes sense.

 A version of the formula is also valid  for (the slightly more general) one dimensional self-adjoint Jacobi operators, and it was subsequently extended in several directions: to band lattice Schr\"odinger operators by Craig and Simon~\cite{CraigSimon}, relating the sum of the nonnegative Lyapunov exponents to the IDS;  to i.i.d. random non self-adjoint Jacobi operators by Goldsheid and Khoruzhenko~\cite{GK2005}; to long-range quasi-periodic Schr\"odinger operators with trigonometric polynomial potentials by Haro and Puig~\cite{HaroPuig}. Finally, a more general version of the formula for self-adjoint block Jacobi matrices with dynamically defined entries was obtained by Chapman and Stolz~\cite{CS15}.
 
Thouless formula was initially employed by Craig and Simon~\cite{CS83, CraigSimon} to establish the $\log$-H\"older continuity of the IDS. 

Since it relates the LE to the IDS via a singular integral operator, Thouless formula can be used to transfer H\"older-type (e.g. H\"older or weak-H\"older) moduli of continuity\footnote{Note that much weaker moduli of continuity, such as $\log$-H\"older, cannot be transferred via the Hilbert transform. This can also be seen by recalling that the IDS is always $\log$-H\"older continuous while the LE can be discontinuous, e.g.  in the case of non-uniformly hyperbolic $\SL(2,\R)$ cocycles in the $C^0$ topology, see~\cite{Bochi02}, or even in the case of quasiperiodic cocycles in the smooth topology, see~\cite{WangYou13,WangYou18}.}  from one quantity to the other, see~\cite[Lemma 10.3]{GS01} for a formal statement. For instance, for the classical Anderson model (where the potential $\{v_n\}_{n\in\Z}$ is an i.i.d. sequence of random variables), Le Page~\cite{LePage} established the H\"older continuity of the LE, which then implies the H\"older continuity of the IDS. This, in turn, can be used in a multiscale analysis scheme to establish the Anderson localization of the operator. This approach was also employed in other related contexts, see for instance~\cite{CS15, DK-Holder}. 

In the opposite direction, the formula can be used to establish {\em limitations} on the modulus of continuity of the LE, via the IDS. This method goes back to Halperin, whose argument was made rigorous by Simon and Taylor~\cite{ST85} and extended by Duarte, Klein and Santos~\cite{DKS19}  and more recently  by Bezerra and Duarte~\cite{BeDu22}.

Moreover,  Thouless formula also plays a role in establishing the absolutely continuous spectrum of the almost Mathieu operator, see Avila~\cite{Avila-ac}.

Note that all of the aforementioned results are within the scope of lattice (or band lattice) Schr\"odinger or Jacobi operators. It turns out that the IDS $N(E)$, a physical quantity, is (linearly) related  to $\rho (A_E)$, the fibered rotation number of the cocycle $A_E$ (the exact linear relation depends on the scaling considered).  

In this paper, instead of the one-parameter family $E \mapsto A_E$ of Schr\"odinger cocycles, we consider  general affine families of $\rm{GL} (2,\R)$ cocycles of the form $A_t = A+tB$ where $A \colon X \to  \rm{GL} (2,\R)$ and $B \colon X \to \Mat(2,\R)$. Under appropriate assumptions, to be formally introduced below, we 
establish the following abstract Thouless formula:
$$
L_1(A_t)=L_1(B)+\int_\R \log\abs{t-s}d\rho(s),\,\,\, \forall t \in \C,
$$
where $L_1$ refers to the first Lyapunov exponent and $d\rho$ is a density measure associated with the fibered rotation number of $A_t$.

\bigskip

 Moreover, we employ the above Thouless formula to establish sharp limitations on the modulus of continuity of the Lyapunov exponent of random linear cocycles, which improve on the result in~\cite{DKS19}. 
 Furthermore, we derive the H\"older-type continuity of the fibered rotation number for linear cocycles over various types of base dynamics.

\bigskip

\subsection{The main assumptions}\label{assumptions}

Let $T \colon X\to X$ be a homeomorphism on a compact metric space $X$ and let $\mu\in\Prob(X)$ 
be an ergodic $T$-invariant probability measure.
Define 
$$\GL_2(\R):=\left\{ A\in \Mat_2(\R) \, \colon \, \det A >0 \, \right\}  $$
to be the group of  $2$ by $2$ matrices with positive determinant.

Consider a continuous function $A\colon \R\times X\to\GL_2(\R)$ which we regard as a one parameter family
$A_t \colon X\to\GL_2(\R)$ indexed by $t\in\R$. Let $F_t\colon X\times \R^2\to X\times \R^2$ be the cocycle defined by
$$F_t(x,v):=(Tx,A_t(x)v),$$
whose iterates are denoted by
$$
A_t^n(x):= A_t(T^{n-1} x)\, \cdots\, A_t(Tx)\, A_t(x) .
$$

\begin{assumption}[Analyticity and Invertibility]\label{inv}
	There are positive constants $R$ and $c$ such that for each $x\in X$, the function $\R\ni t \mapsto A_t(x)$ admits an analytic extension to the complex strip
	$$ \mathscr{S}_R:= \{ z\in \C\colon\,  |\Imp z|\leq R \}   $$
	 with  $\abs{\det(A_t(x))}\geq c>0$
	for all $(t,x)\in \mathscr{S}_R\times X$. 
\end{assumption}

Many  of the results below are stated for one-parameter families of matrices $\{A_t  \colon  t\in I\}$ which are smooth but not necessarily analytic, where the index set $I$  always stands for some interval $I\subset\R$.

\begin{definition}
	\label{winding def}
	A smooth curve  of matrices $I\ni t\mapsto A_t\in\GL_2(\R)$ is said to be positively (resp. negatively)  winding, if for all $t\in I$, the quadratic form
	$Q_{A_t}:\R^2\to\R$,
	$$Q_{A_t}(v):= (A_t \,v)\wedge (\dot A_t\, v) = (\det A_t)\,  v \wedge (A_t^{-1} \, \dot A_t \, v),$$
	is  positive (resp. negative) semidefinite, with one eigenvalue bounded away from $0$. Here $\dot A_t:=\frac{d}{dt}A_t$ and given any two vectors $v_1, v_2\in\R^2$,
	$v_1\wedge v_2 := \det(v_1, v_2)=\norm{v_1} \norm{v_2} \sin \measuredangle(v_1, v_2)$.
\end{definition}

Positive (negative) winding means that for every  non-zero vector $v\in\R^2$ which is not a real eigenvector of any of the matrices $A_t^{-1} \, \dot A_t$, the curve $I\ni t\mapsto A_t\, v\in\R^2\setminus\{0\}$ winds positively (resp. negatively) around the origin as
$t$ runs in $I$.

\begin{definition}	
	A family of cocycles $A_t:X\to\GL_2(\R)$ is said to be positively (negatively)  winding, if for every $x\in X$ the analytic curve $I\ni t\mapsto A_t(x)$ is   positively (negatively)  winding.
%	$Q_{A_t(x)}(v)\geq 0$ for every $(t,x,\hat v)\in I\times X\times\Pp$ and moreover 
%	for all $x\in X$ the smooth curve $I\ni t \mapsto A_t(x)\in\GL_2(\R)$ is positively (negatively)  winding.
\end{definition}

\begin{assumption}[Winding]
	\label{winding assumption}
	The family of cocycles $A_t$ is positively (or negatively)  winding. 
	%Moreover, there exist unit vectors $v\in\R^2$ such that $Q_{x,t}(v)$ is uniformly bounded away from $0$ for all $(x,t)\in X\times I$.
\end{assumption}

The first Lyapunov exponent of the cocycle, denoted by  $L_1(A_t)$, measures the fiber exponential growth rate along the orbits of $F_t$. Since $(T,\mu)$ is ergodic, by  J. Kingman sub-additive theorem~\cite{Ki68}, for $\mu$-almost every $x\in X$
$$ L_1(A_t) =\lim_{n\to +\infty} \frac{1}{n}\,\log \norm{A_t^n(x)}.$$
%If we consider $\mu$ as being ergodic, the Lyapunov exponent is constant in a full measure set of $X$ and we denote it by $L_1(A_t)$

Throughout the manuscript  we denote by $\proj{v}\in\Pp^1$ the projective point of a vector $v \in \R^2\setminus\{0\}$.
Similarly we denote by $\hat A$ the projective action of a matrix $A\in\SL_2(\R)$ and
let $\hat{F}_t:X\times \Pp^1\to X\times\Pp^1$ denote the projective cocycle $$\hat F_t(x, \proj v):=(T x, \proj{A}_t(x) \proj{v}) .$$
For the sake of notational simplicity, many times we write $A\hat v$ instead of $\hat A \hat v$.
The fibered rotation number of $A_t$, denoted by $\rho(A_t)$, is defined as the $\mu$-almost sure limit
\begin{equation}\label{def rot number}
	\rho(A_t) :=  \lim_{n\to +\infty} \frac{1}{\pi n}\, \measuredangle(A_t^n(x)\hat v, \hat v)  , 
\end{equation}
where $(x,t,v)\in X\times  I \times \Pp^1$.
In Section~\ref{winding} we properly define the angle $\measuredangle(A_t^n(x)\hat v, \hat v)$ and show that $I\ni t\mapsto \rho(A_t)$ is a well defined, non-decreasing and continuous function.
Moreover up to an additive constant the fibered rotation number is independent of the choices made to define the angle   $\measuredangle(A_t^n(x)\hat v, \hat v)$.

\begin{assumption}[Affine form]\label{affine}
	The one-parameter family $A_t$ has the form $A_t(x)=A(x)+tB(x)$ where $A  \colon X\to \GL_2(\R)$  and $B \colon X\to \Mat_2(\R)$ are continuous functions.
\end{assumption}

The family of cocycles $A_t$ is well defined for all $t\in\C$, although, by  Assumption 1, the matrices $A_t(x)$ are  possibly only invertible for $t\in \mathscr{S}_{R}$.

We say  that $B$  has \emph{dominated splitting} when there exists  a continuous decomposition  $\R^2=E_0(x)\oplus E_\infty(x)$ in lines $E_0(x)$ and $E_\infty(x)$, which is $T$-invariant, i.e.,
 $B(x)  E_0(x)=E_0(T x)$ and $B(x) E_\infty(x)\subset E_\infty(T x)$, for all $x\in X$, and such that
for some integer $n_0$, 
$\norm{ B^{n_0}(x)\, v_0} > \norm{ B^{n_0}(x)\, v_\infty} $
for all $x\in X$ and all unit vectors $v_0\in E_0(x)$ and $v_\infty \in E_\infty(x)$. For a rank $1$ cocycle $B$, $E_\infty(x)=\Ker(B(x))$.

\begin{assumption}[Dominated Splitting]\label{dom-split}
	The cocycle $B$ has dominated splitting. In particular we have that $L_1(B)>L_2(B)\geq -\infty$. 
\end{assumption}

\subsection{Statements}
\label{statements}
We can now state the main result of this paper and some of its consequences.

\begin{theorem}\label{main}
	With assumptions 1-4 fulfilled, for any $t\in \C$,
	\begin{equation}
		\label{Thouless ID}
		L_1(A_t)=L_1(B)+\int_\R \log\abs{t-s}d\rho(s) ,
	\end{equation}
	where $d\rho$ is the Lebesgue-Stieltjes measure associated with the fibered rotation number $\rho(A_t)$.
\end{theorem}

\begin{remark}
\label{rmk assumption 4}

The dominated splitting assumption is only used in Lemma~\ref{compact support} below through the following chain of implications
$$\begin{aligned}
 B\text{ has dominated splitting }&\Rightarrow \begin{aligned}  &A_t \text{ has dominated}\text{ splitting }
\forall\text{ large }t \end{aligned}\\
& \Rightarrow  d\rho\text{ has compact support}.
\end{aligned}$$
If we can prove that $A_t$ has dominated splitting for all sufficiently large $t$, or else that the Lebesgue-Stieltjes measure $d\rho$ has compact support then Theorem~\ref{main} would hold with Assumption 4 replaced by the much weaker hypothesis $L_1(B)>L_2(B)\geq -\infty$.

Moreover,  the compactness of $\supp(d\rho)$ is a technical assumption required in the proof of
Proposition~\ref{lim sum f(tk)=int f drho}. So it is possible that this proposition, and whence Theorem~\ref{main}, still holds even if this support is not compact.
\end{remark}

As a consequence of Theorem~\ref{main} we establish a limitation on the modulus of continuity of the Lyapunov exponent of {\em random} linear cocycles. Let
$X:=\{1,\ldots, \kappa\}^\Z$ be the space of sequences in $\kappa$ symbols, and let  $T:X\to X$ be the Bernoulli shift on $X$ equipped  with some Bernoulli probability  measure $\mu=(p_1,\ldots, p_\kappa)^\Z$, where $p_1+\cdots + p_\kappa=1$ and $p_j>0$ for $j=1,\ldots, \kappa$. 

A random or locally constant cocycle $\bld A \colon X\to \SL_2(\R)$ is determined by
a vector of $\kappa$ matrices $\underline A=(A_1,\ldots, A_\kappa)\in \SL_2(\R)^\kappa$, via the formula $\bld A(\omega):=A_{\omega_0}$,
where $\omega=(\omega_j)_{j\in\Z} \in X$. We will use the notations $\underline A$ and $\bld A$ interchangeably. 

The iterates of the cocycle $\bld A$ are thus the multiplicative random process corresponding to the finitely supported measure
$$  \mu(\underline A):=\sum_{j=1}^\kappa p_j\, \delta_{A_j} \in\Prob(\SL_2(\R)) .$$
Denote by $H(\mu):=-\sum_{j=1}^\kappa p_j\, \log p_j$ the Shannon entropy of the measure $\mu(\underline A)$. A simplified version of our result is as follows (see Corollary~\ref{regularity dichotomy} for its more precise formulation). Throughout the manuscript we will use the notation
$\SL_2^\ast(\R):=\SL_2(\R)\setminus\{-I,I\}$.

\begin{theorem}\label{intro low regularity thm}

Let $\underline A=(A_1,\ldots, A_\kappa)\in \SL_2^\ast(\R)^\kappa$ be a random cocycle  such that $L_1 (\underline A) > 0$ but $\underline A$ is not uniformly hyperbolic. There is an open set $\U$ of ``directions'' in $\{ \underline P = (P_1, \ldots, P_\kappa) \colon   P_i^2 = 0, \tr P_i = 0 \ \forall i \} \subset \Mat_2 (\R)^\kappa$, such that for all  $ \underline P = (P_1, \ldots, P_\kappa) \in \U$, if we denote by $\bld A_t$ the random cocycle determined by the list $\underline A_t := \underline A \, (I + t \underline P)$, then the map $t \mapsto L_1 (\bld A_t)$ is not $\alpha$-H\"older continuous near $t=0$ provided that $\alpha > \alpha (\underline A) := \frac{H(\mu)}{L_1 (\underline A)} $.
\end{theorem} 

	 This result extends~\cite[Theorem A]{BeDu22},
	 where the same conclusion was achieved
	for a particular curve of cocycles obtained embedding $A$ into a family of Schr\"odinger  cocycles over a Markov shift.

%COMMENTS ABOUT THE DIAGONAL CASE %%%%%%%%%%%%%%%%%
% If a random cocycle $\underline A=(A_1,\ldots, A_\kappa)\in \SL_2(\R)^\kappa$ is diagonalizable (in the sense that its components $A_i$, $i=1, \ldots, \kappa$ are simultaneously diagonalizable) then its random walk entropy is zero. We then obtain the following corollary.

% \begin{corollary}
% Let $\underline A=(A_1,\ldots, A_\kappa)\in \SL_2(\R)^\kappa$ be a diagonalizable cocycle such that $L_1 (\underline A) > 0$ but $\underline A$ is not uniformly hyperbolic. Then the map $t \mapsto L_1 (A_t)$ cannot be H\"older continuous near $t=0$.
% \end{corollary}

% \begin{remark}
% In particular, the iterates of a diagonalizable cocycle $\underline A$ cannot satisfy uniform large deviations estimates of exponential type, as this would imply, via an abstract continuity theorem, the H\"older continuity of the Lyapunov exponent in a neighborhood of $\underline A$ (see~\cite{DK-book, DK-CBM} for more details). This shows that the sub-exponential uniform large deviations obtained in~\cite{DK-Holder} are nearly optimal.
% \end{remark}
%%%%%%%%%%%%%%%%%%%%%%%%%%%%%%%%%%%%%%%%%

\begin{remark}

It is well known, starting with the work of Le Page~\cite{LePage}, see also Duarte and Klein~\cite{DK-book, DK-CBM} that under an irreducibility assumption, the maximal Lyapunov exponent is always H\"older continuous near cocycles  $\bld A$ with $L_1 (\bld A) > 0$. Moreover, in~\cite{DK-Holder} it was shown that if such a cocycle $\bld A$ is diagonalizable,  then the Lyapunov exponent is at least weak-H\"older continuous in its vicinity, and if it is not diagonalizable (since then either $\bld A$ or its inverse satisfy some irreducibility condition), the Lyapunov exponent is locally H\"older continuous.

Moreover, since uniform hyperbolicity is an open property, if a cocycle $\bld A$ is uniformly hyperbolic, then the same holds for 
$\bld A_t$ near $t=0$. In this case  the map $t \mapsto L_1 (\bld A_t)$ is analytic. 

Therefore we have the following dichotomy for an irreducible  cocycle $\bld A$ with $L_1 (\bld A) > 0$. Near $t=0$, the Lyapunov exponent of $\bld A_t$ is either analytic or it is continuous with a strict limitation on the strength of its modulus of continuity, namely  $\alpha$-H\"older with $\alpha \le \alpha (\underline A)$.
\end{remark}

Thouless formula~\eqref{Thouless ID} shows that the maps $t \mapsto L_1(A_t) - L_1(B)$ and $t \mapsto \rho (A_t)$ are obtained one from the other via the Hilbert transform (see for instance~\cite{Duoan-FA} for its definition). It is well known (see~\cite[Lemma 10.3]{GS01}) that the Hilbert transform preserves H\"older-type continuity properties. 

The H\"older (or weak-H\"older) continuity of the maximal Lyapunov exponents has been established for various types of linear cocycles. We list only a few of the more recent such results (for a more complete list of results, see the references therein). Lyapunov exponents of quasiperiodic cocycles (that is, linear cocycles over a torus translation) are H\"older or weak-H\"older continuous provided the translation frequency satisfies an appropriate arithmetic condition (e.g. a Diophantine condition) and the fiber action depends analytically on the base point (see~\cite{DK-Coposim}); results in other regimes (e.g. almost reducibility) are also available, see for instance~\cite{CCYZ}. Lyapunov exponents of random cocycles (i.e. locally constant cocycles over a Bernoulli or Markov shift) are H\"older continuous assuming a generic irreducibility condition (see~\cite{LePage} and~\cite[Chapter 5]{DK-book}) and weak-H\"older continuous without such an assumption (see~\cite{DK-Holder}). Under appropriate conditions, Lyapunov exponents of linear cocycles over uniformly hyperbolic systems are H\"older continuous (see~\cite{DKP}). Similar continuity properties were also obtained for mixed random-quasiperiodic cocycles (see~\cite{CDK-paper4}). 

\begin{proposition}\label{conseq prop} 
For each of the types of linear cocycles described above, under the specific assumptions ensuring the H\"older (resp. weak-H\"older) continuity of the Lyapunov exponent, the fibered rotation number $\rho (A_t)$,	 of a family of cocycles $A_t$ satisfying assumptions 1-4 above, is a locally H\"older (resp. locally weak-H\"older) continuous function.
\end{proposition}

It is natural to consider the problem of extending the results in this paper in other  directions, as follows: relaxing Assumption~\ref{dom-split} as explained in Remark~\ref{rmk assumption 4}; considering instead of affine one parameter families $A_t$, polynomial or even more general families;  obtaining an analogue of Thouless formula in Theorem~\ref{main} for symplectic (higher dimensional) cocycles; using the approach in Theorem~\ref{intro low regularity thm} to derive limitations on the modulus of continuity for other types of cocycles, such as mixed random-quasiperiodic cocycles, see~\cite{CDK-paper1, Jamerson-M}. These extensions will be considered in separate projects.

In~\cite{GK21} A. Gorodetski and V. Kleptsyn have established the following result
under a similar setting.
Given a smooth family   of positively winding  cocycles
$\{A_t:X\to \SL(2,\R)\}_{t\in I}$, which is not uniformly hyperbolic for any $t\in I$, there exist  $\Omega\subset X$ with full probability and a residual subset  $S\subset I$ such that 
\begin{align*}
		\limsup_{n\to \infty} \frac{1}{n}\, \log \norm{A_t^n(x)} = L_1(A_t) \qquad &\forall x\in \Omega,\; t\in I \\
		\liminf_{n\to \infty} \frac{1}{n}\, \log \norm{A_t^n(x)} = 0 \qquad &\forall x\in \Omega,\; t\in S  .
\end{align*}
This behavior complements the regularity dichotomy  in Theorem~\ref{intro low regularity thm}. It would  be interesting to explore the connection between the orbits analyzed in~\cite{GK21} and the fractal structure of the set of heteroclinic tangencies alluded to in Subsection 1.3 of ~\cite{BeDu22}.

Our results are also related to the work~\cite{DD2012},
where B. Deroin and R. Dujardin
study holomorphic families $\{A_\lambda\}_{\lambda\in\Lambda}$  of random (locally constant) $\SL(2,\C)$-cocycles
parametrized on a complex manifold $\Lambda$. The connection becomes more clear when $\Lambda$ has dimension $1$, say $\Lambda=\C$. In this case the Lyapunov exponent $\lambda\mapsto L_1(A_\lambda)$ is a subharmonic function  and its Laplacian (in the sense of distributions) is the so-called \textit{bifurcation current} $T_{\mathrm{bif}}$. In~\cite[Theorem A]{DD2012}
the authors characterize  the \textit{bifurcation locus}
of the holomorphic family $\{A_\lambda\}_{\lambda\in\Lambda}$
as being the support of the current $T_{\mathrm{bif}}$.
For positively winding families, the bifurcation locus coincides with the set of parameters $\lambda\in\Lambda$ such that $A_\lambda$ is not uniformly hyperbolic. Moreover,~\cite[Theorem 3.5]{DD2012} can  be used to prove that this  bifurcation current matches the fibered rotation measure, i.e., 
$T_{\mathrm{bif}}=d\rho$. The Thouless formula~\eqref{Thouless ID}  follows then from the Riesz representation theorem for subharmonic functions (see~\cite{HK1976}). This connection provides an enlightening description of $d\rho$ and its support in the special case where the homeomorphism $T:X\to X$ is a Bernoulli shift in finitely many symbols.

\medskip

The rest of this paper is organized as follows. In Chapter~\ref{winding} we study the winding property and formally introduce the fibered rotation number. In Chapter~\ref{mainproof} we provide the proof of our main result, the Thouless formula in Theorem~\ref{main}. In Chapter~\ref{affine families} we present sufficient conditions ensuring the validity of each of the main assumptions besides the affine form (namely the invertibility, the winding property and the dominated splitting). In Chapter~\ref{consequences} we present applications of the main result, namely we establish Theorem~\ref{intro low regularity thm} regarding  the limitations on the modulus of continuity
of random cocycles. In the Appendix, Chapter 6, we develop  some linear algebra tools used in the rest of the work.

\section{The winding property}\label{winding}

In this section we establish some consequences of the winding property in Assumption \ref{winding assumption}. In Sections 3.1 and 3.2 the results are stated for general the one parameter differentiable family of maps $A_t\colon X \to \GL_2(\mathbb{R})$ indexed by $t\in I$. Section 3.3 is centred in the particular case where $A_t=A+tB$.

\subsection{General properties}

The winding property means that for almost all $\hat v\in\Pp^1$, the projective curve $I\ni t\mapsto \hat A_t\,\hat v\in\Pp^1$, has non-vanishing derivative for almost every $t\in I$,
keeping the same orientation as $t$ runs in $I$. 
See Proposition~\ref{winding meaning} below.

\begin{proposition}
	\label{derivative At v/norm}
	Given a unit vector $ v\in \R^2$,
	$$   \frac{d}{dt}[\,  A_t\, \hat v \, ] = \frac{d}{dt} \frac{A_t\, v}{\norm{A_t\, v}}      =  
	\frac{ (A_t\, v)\wedge (\dot A_t\, v)}{\norm{A_t\, v}^2}   $$
	where as before $\dot A_t= \frac{d}{dt}A_t$.
\end{proposition}

\begin{remark}
To interpret this equality the reader should
either regard right-hand-side derivative as a real number, because  $\Pp^1$ is $1$-dimensional, or else the left-hand-side as a vector, multiplying it by the unique unit and  positive vector in $T_{A_t v/\norm{A_t v}}\Pp^1$.
\end{remark}

\begin{proof}
We want to establish an expression for the variation of the angle of the projective map $t\mapsto \hat A_t \hat v$. Thus, since the metric considered in $\Pp^1$ is $d(\hat{v},\hat{w}):=\frac{|v\wedge w|}{\|v\|\|w\|}$, we get that,
$$d(\proj A_t \proj v, \proj A_{t^\prime} \proj v)=\frac{| A_t v\wedge A_{t^\prime}v|}{\|A_tv\| \|A_{t^{\prime}}v\|}=\frac{| A_t v\wedge (A_{t^\prime}-A_t)v|}{\|A_tv\| \|A_{t^{\prime}}v\|}$$
 by adding $A_t v$ in the second term because $A_tv\wedge A_tv=0$. Dividing by $|t-t^{\prime}|$,
$$\frac{d(\proj A_t \proj v, \proj A_{t^\prime} \proj v)}{|t-t^\prime|}=\frac{| A_t v\wedge \frac{A_{t^\prime}-A_t}{t-t^\prime}v|}{\|A_tv\| \|A_{t^{\prime}}v\|}$$
and since the limit on the left hand side when $t^\prime$ goes to $t$ is the absolute value of the derivative we obtain $\left\vert\frac{d}{dt}  \frac{A_t\, v}{\norm{A_t\, v}}  \right\vert=
	\left\vert\frac{ (A_t\, v)\wedge (\dot A_t\, v)}{\norm{A_t\, v}^2} \,\right\vert$.
	The identity follows from simple geometric considerations on the oriented angle between $A_t v$ and $\dot A_t v$.
\end{proof}

\begin{proposition}
	\label{winding meaning}
	Let $I\ni t\mapsto A_t\in \GL_2(\R)$ be an analytic  curve with the winding property. For almost every $\hat v\in\Pp^1$ and almost every $t\in I$, the map
	$I\ni t\mapsto \hat A_t\,\hat v\in\Pp^1$
	has non-vanishing derivative and
	keeps the same orientation as $t$ runs in $I$. 
\end{proposition}

\begin{proof}
By the winding assumption, $Q_{A_t}$ is either positive or negative semi-definite for all $t$. From now on we assume the winding is positive.
By Definition~\ref{winding def}, 
$Q_{A_t} (v):=(\det A_t)\, v\wedge A_t^{-1}\dot A_t v=0$
if and only if $v$ is an eigenvector of $A_t^{-1}\dot A_t$, and since $Q_{A_t}$ is a quadratic form, this eigenvector is unique. Therefore, if we denote  
$$\mathrm{Eig}(A_t):=\left\{\hat v\in\Pp^1 \colon v \text{ is an eigenvector of } A_t^{-1}\dot A_t \right\},$$
then either $Q_{A_t}$ is positive definite and $\mathrm{Eig}(A_t)=\emptyset$, or else
$Q_{A_t}$ is positive semi-definite and $\mathrm{Eig}(A_t)$ is singleton.

To follow we analyse the cases where $\mathrm{Eig}(A_t)\not=\emptyset$, otherwise the conclusion is obvious. Therefore for the degree two polynomial  $p(\lambda):=\det(I + \lambda \,A_t^{-1}\dot A_t)$, its discriminant is given by $\Delta(t):=4\,\det( A_t^{-1}\dot A_t))-\tr (A_t^{-1}\dot A_t)^2$ 
 up to a sign, which is proved in Proposition~\ref{prop invertibility criterion} below.
Because this function is analytic, we consider two cases

In the first case, $\Delta(t)\not\equiv 0$ and $Z:=\{t\in I \colon \Delta(t)=0 \}$
is a countable set consisting of isolated points. By Proposition~\ref{derivative At v/norm} for all $t\notin Z$, $\frac{d}{dt}  \frac{A_t\, v}{\norm{A_t\, v}}=\norm{A_t v}^{-2} Q_{A_t}(v) >0$ and the conclusion follows.

In the second case $\Delta(t)\equiv 0$,  $\mathrm{Eig}(A_t)\neq \emptyset$  for all $t\in I$,
and there exists $\varphi:I\to \Pp^1$ analytic such that
$\mathrm{Eig}(A_t)=\{\varphi(t)\}$ for all $t\in I$.
Let $V$ be the set of regular values of   $\varphi:I\to \Pp^1$. By Sard's Theorem, $V$ has full  measure in $\Pp^1$. For $\hat v\in V$ the set $Z_{\hat v}:=\{t\in I\colon \hat v\in \mathrm{Eig}(A_t) \}
= \{ t\in I\colon \varphi(t)=\hat v\}$ is again a countable set consisting of  isolated regular points.
Since for all $t\notin Z_{\hat{v}}$, $\frac{d}{dt}  \frac{A_t\, v}{\norm{A_t\, v}}=\norm{A_t v}^{-2} Q_{A_t}(v) >0$,
 the map
$I\ni t\mapsto \hat A_t\,\hat v\in\Pp^1$
has non-vanishing derivative for almost every $t\in I$,
keeping  the same orientation as $t$ runs in $I$.
\end{proof}

We want to see that the winding property 
is preserved under  iterations of the cocycle. The following proposition  will imply that the composition of positively (negatively) winding cocycles is also positively (negatively) winding.

\begin{proposition}
	\label{prods of pos winding are pos winding}
	Given curves $I\ni t\mapsto A_{i,t}\in \GL_2(\R)$, for $i=1,\ldots, n$, if each $A_{i,t}$ is positively (negatively) winding then so is their product
	$M_t:= A_{n,t}\, \cdots \, A_{2,t}\, A_{1,t}$.
\end{proposition}

\begin{proof}
	For simplicity consider $n=2$.
	Writing $v_1 :=  A_{1,t}\, v$, we have
	\begin{align*}
	(M_t\, v) \wedge (\dot M_t\, v) &=
	(A_{2,t}\, A_{1,t}\, v) \wedge (\dot A_{2,t}\, A_{1,t}\, v + A_{2,t}\, \dot A_{1,t}\, v) \\
	&= (A_{2,t}\, v_1) \wedge (\dot A_{2,t}\,v_1) +
	(A_{2,t}\, A_{1,t}\, v) \wedge (A_{2,t}\, \dot A_{1,t}\, v)\\
	&= \underbrace{ (A_{2,t}\, v_1) \wedge (\dot A_{2,t}\,v_1)  }_{\geq0} +
	(\det A_{2,t})\,\underbrace{ ( A_{1,t}\, v) \wedge ( \dot A_{1,t}\, v) }_{\geq0}  \geq0.
	\end{align*}	
Moreover, it is positive if for example $v$ is not an eigenvector of $A_{1,t}^{-1}\dot A_{1,t}$. Thus $M_t$ is positively winding. The general case follows by induction.
\end{proof}

The previous argument shows a bit more.
For the sake of simplicity we only state the following result for $\SL_2(\R)$-valued curves.

\begin{proposition} 	
	\label{winding velocity formula}
	Given curves $I\ni t\mapsto A_{i,t}\in \SL_2(\R)$, for $i=1,\ldots, n$,
	define $\vb_j(t):=A_{j,t}\, \cdots\, A_{1,t}\, v/ \norm{ A_{j,t}\, \cdots\, A_{1,t}\, v }$, with the convention that $\vb_0=v$.
	Then for $M_t:= A_{n,t}\, \cdots \, A_{2,t}\, A_{1,t}$, 
	\begin{align*}
	\frac{ (M_t v)\wedge (\dot M_t v)}{\norm{M_t v}^2} &= \sum_{j=1}^n
	\frac{1}{\norm{A_{n,t}\,\cdots \,  A_{j+1,t} \, \vb_j}^2}\, \frac{ (A_{j,t} \vb_{j-1} )\wedge (\dot A_{j,t} \vb_{j-1})}{\norm{A_{j,t} \vb_{j-1}}^2}\\
	 &= \sum_{j=1}^n
	\left( \frac{\norm{A_{j,t}\, \cdots \, A_{1,t}\, v}}{\norm{M_t\, v}} \right)^2\, \frac{ (A_{j,t} \vb_{j-1} )\wedge (\dot A_{j,t} \vb_{j-1})}{\norm{A_{j,t} \vb_{j-1}}^2} .	
	\end{align*} 
\end{proposition}

\begin{remark}
	The ratio $\frac{ (M_t v)\wedge (\dot M_t v)}{\norm{\hat M_t \hat v}^2}$ measures the rotation speed of $M_t v$. Likewise the ratio
	$\frac{ (A_{j,t} \vb_{j-1} )\wedge (\dot A_{j,t} \vb_{j-1})}{\norm{A_{j,t} \vb_{j-1}}^2}$ measures the rotation speed of $\hat A_{j,t} \hat w$ when  $w:=\vb_{j-1}(t)$ is fixed.
\end{remark}

\begin{corollary}
	\label{coro:pos winding cocycle}
	If $A_t:X\to\GL_2(\R)$ is a  positively (negatively) winding  family of cocycles then for every $n\in\N$ the family of iterated cocycles
	$A^n_t:X\to\GL_2(\R)$ is positively (negatively) winding.
\end{corollary}

\bigskip

\subsection{Fibered Rotation Number}
In this section we introduce the notion of fibered rotation number referred to in the main theorem. This concept  was first introduced in~\cite{He83b} and  further developed in~\cite{ABD12,GK21}.

Let $\pi\colon \R\to \Pp^1$
denote the canonical covering map of  $\Pp^1$, which induces a diffeomorphism between $\T^1:=\R/\pi\Z$ and $\Pp^1$. 
Consider a continuous family of cocycles $A_t:X\to \GL_2(\R)$, with parameter $t\in I$, over the continuous base map $T:X\to X$.
Each matrix $A_t(x)$ admits a lifting $\tilde A_t(x):\R\to\R$ such that
the following diagram commutes
$$\begin{CD}
	\R @>\tilde A_t(x) >> \R\\
	@V\pi VV @VV\pi V\\
	\Pp^1 @>> \hat A_t(x) > \Pp^1
\end{CD} \; . $$
We gather these liftings in a single function
$\tilde G:X\times I\times\R\to\R$,
$\tilde G(x,t,v):=\tilde A_t(x)(v)$.

Given two non-zero vectors $v,w\in\R^2$, the angle
$\angle(v, w)$ is well defined as an element in $\R/(2\pi\Z)$. We will use the notation $\measuredangle(v,w)$ to represent a real argument of $\angle(v, w)$ so that
$$ \angle(v, w) = \measuredangle(v,w) + 2\pi\,\Z . $$
Notice that $\measuredangle(v,w)$ can not be globally and continuously defined as a function of $(v,w)\in \Su^1\times \Su^1$, where $\Su^1:=\{v\in\R^2\colon \norm{v}=1\}$.

\begin{proposition}
	\label{prop def h}
	Fix  $t_0\in I$
	and a unit vector $v_0\in\Su^1$. Then there exists $h:X\to\R$ such that 
	\begin{enumerate}
		\item $h$ is bounded and measurable,
		\item $\angle(A_{t_0}(x)\, v_0, v_0)=h(x)+2\pi\Z$ for all $x\in X$,
		\item  $\mu\left\{x\in X\colon x\, \text{ is a discontinuity point of }\, h \, \right\}=0$.
	\end{enumerate}
\end{proposition}

\begin{proof}
	For each point $x\in X$ we can take a radius $r>0$ such that the ball $B_r(x):=\{z\in X\colon d(z,x)<r\}$ has  boundary
	$\partial B_r(x):=\{z\in X\colon d(z,x)=r\}$ with zero measure,
	$\mu (\partial B_r(x))=0$, and a locally defined continuous function
	$h_x:B_r(x)\to\R$ such that 
	$\angle(A_{t_0}(z)\, v_0, v_0)=h_x(z)+2\pi\Z$ for all $z\in B_r(x)$.  
	Since $X$ is compact we can cover $X$ with a finite number of these balls $B_1,\ldots, B_m$, where
	$B_i=B_{r_i}(x_i)$ for $i=1,\ldots, m$. Writing $h_i=h_{x_i}$, we define
	$h(x):=h_1(x)$ if $x\in B_1$ and more generally
	$h(x):=h_i(x)$ if $x\in B_i\setminus (B_1\cup \cdots \cup B_{i-1})$ for some $i=2,\ldots, m$.
	This function $h$ satisfies (1)-(3).
\end{proof}

\begin{proposition}
\label{tilde G}
The function $\tilde G$, dependent on the arbitrary choices of the liftings $\tilde A_t(x)$, can be made continuous in $(t,v)\in I\times \R$, and measurable in $(x,t,v)\in X\times I\times \R$ in a way that for some measurable set $D\subset X$ with $\mu(D)=0$, the function $\tilde G$ is continuous over  $(X\setminus D)\times I\times\R$.
\end{proposition}

\begin{proof}
In general it may not be possible to realize $\tilde G$ as a globally continuous function, see~\cite[Remark A.4]{GK21}.
Fix $t_0\in\R$, $v_0\in\Su^1$, $x_0\in\R$ such that $\pi(x_0)=\hat v_0$   and a measurable function $h:X\to\R$ as in Proposition~\ref{prop def h}. For each $x\in X$
let $\tilde A_{t_0}(x)\colon\R\to\R$ be the unique lifting of
$\hat A_{t_0}(x):\Pp^1\to\Pp^1$ such that
$ \tilde A_{t_0}(x)(x_0) - x_0 = h(x)$.
Extending these liftings continuously in the variable $t\in I$, we obtain a function $\tilde G(x,t,v):=\tilde A_t(x)\, v-v$, continuous in $(t,v)\in I\times\R$, and sharing with $h(x)$ the same continuity points in $X$.
\end{proof}

 We define recursively 
$\tilde G_n:X\times I\times \R\to\R$,
$$  \tilde G_n(x,t,v):= \tilde G\left(T^{n-1} x,t, \tilde G_{n-1}(x,t,v) \right) \; \text{ for }  n \geq 1 $$
with $\tilde G_0(x,t,v):=v$.
Through this function we define the angle
$$ \measuredangle(A_t^n(x)\, \pi(v), \pi(v) ):= \tilde G_n(x,t,v)-v .$$
The left-hand-side does not depend on the representative $v\in\R$ of the projective point $\pi(v)\in\Pp^1$.
Hence the expression $\measuredangle(A_t^n(x) v, v)=\measuredangle(A_t^n(x)\hat v, \hat v)$ makes sense and defines a function on $X\times I\times \Pp^1$.
For $n=1$, this expression determines the family of functions
$\tilde H_t\colon X \times \Pp^1\to\R$, 
$$ \tilde H_t(x,\hat v ):= \measuredangle(A_t(x)\, \hat v, \hat v) . $$

Given $s<t$, define for any $(x,v)\in X\times \Pp^1$,
$$ \measuredangle(A_t^n(x)\, v, \, A_s^n(x)\, v):=
 \measuredangle(A_t^n(x)\, v, \, v) -  \measuredangle(A_s^n(x)\, v, \, v) . $$	
\begin{proposition}
	\label{Hn properties}
	The functions defined above satisfy:
	\begin{enumerate}
		\item  For each $x\in X$, 
		$I\times\Pp^1 \ni (t,\hat v)\mapsto \measuredangle(A_t^n(x)\, \hat v, \hat v)$, is a continuous map;
		\item  $\measuredangle(A_{t}^n(x)\,  \hat v,  \hat v)  =\sum_{j=0}^{n-1} \tilde H_t(\hat F^j(x, \hat  v))$,\;  $\forall t\in I$ $\forall (x, \hat v) \in   X \times \Pp^1$;
		\item The angle functions $(x,t,\hat v)\mapsto \measuredangle(A_{t}^n(x)\, \hat v, \hat v)$ are continuous in the complement of a product set $D\times I\times \Pp^1$ where $\mu(D)=0$;
		\item  $ \angle(A_t^n(x)\, v, \, A_s^n(x)\, v)
		=  \measuredangle(A_t^n(x)\, v, \, A_s^n(x)\, v)+ \pi\,\Z$;
		\item  $\measuredangle(A_t^n(x)\, v, \, A_s^n(x)\, v)\geq 0$, \; for $t\geq s$,  if the family of cocycles $A_t$ is positively winding. 
	
	\end{enumerate}
\end{proposition}

\begin{proof}
Follows from the definitions and Proposition~\ref{tilde G}.
The set in item (3) is the union $D:=\cup_{j\in \Z} T^{-j}(D_h)$, where $D_h$ is the set of discontinuity points of the  function $h$
used in the proof of Proposition~\ref{tilde G}.

For item (5),
notice that	
\begin{align*}
\measuredangle(A_t^n(x)\, v, \, A_s^n(x)\, v)&:=
\measuredangle(A_t^n(x)\, v, \, v)-\measuredangle(A_s^n(x)\, v, \, \, v)\\
&=\sum_{j=0}^{n-1} \tilde H_t(F^j(x,\hat v))
- \tilde H_s(F^j(x,\hat v))  
\end{align*}
and the function $\tilde H_t(x,\hat v)$ is non-decreasing in the variable $t$, as a consequence of the winding property.
\end{proof}

\begin{proposition}
	\label{rotation number full measure set}
	For every $t\in I$ there exist a number $\rho\in\R$ and a measurable set  $\Omega_t\subset X$ with full measure, $\mu(\Omega_t)=1$, such that for all $x\in \Omega_t$ and all $\hat v\in\Pp^1$,
	$$  \rho =  \lim_{n\to +\infty} \frac{1}{\pi n}\, \measuredangle(A_t^n(x)\, \hat v, \hat v) .  $$
\end{proposition}
\begin{proof}
Follows from~\cite[Proposition A.1]{GK21}.
\end{proof}

\begin{definition}
The previous limit $\rho$ is called the fibered rotation number of the cocycle $A_t$ and denoted by $\rho(A_t)$.
\end{definition}

\begin{proposition}
\label{rot number via stationary measures}
Given $t\in I$, for any measure $\nu\in \Prob(X\times\Pp^1)$ such that $(\hat F_t)_\ast\nu=\nu$
and $\pi_\ast\nu=\mu$, where $\pi\colon X\times\Pp^1\to X$
is the canonical projection $\pi(x,\hat v):=x$, we have
$$ \rho(A_t) =\frac{1}{\pi} \, \int_{ X\times \Pp^1} \measuredangle(A_t(x)\, \hat v, \hat v)\, d\nu(x,\hat v) .$$
\end{proposition}

\begin{proof}
Assuming $(\hat F_t, \nu)$ is ergodic, the conclusion follows applying  Birk\-hoff ergodic theorem to item (2) of Proposition~\ref{Hn properties} and Proposition~\ref{rotation number full measure set}. In general the space $M$ of
all measures $\nu\in \Prob(X\times\Pp^1)$ such that $(\hat F_t)_\ast\nu =\nu$
and $\pi_\ast\nu=\mu$ is weak$\ast$ compact and convex. Moreover the extremal points of $M$ are the measures $\nu\in M$ such that $(\hat F_t, \nu)$ is ergodic. 
When $\nu$ is not ergodic, it admits an ergodic decomposition consisting of ergodic measures, i.e., extremal points of $M$. Hence the stated  identity must hold in this case as well.
\end{proof}

\begin{proposition}
	\label{rho continuity and monotonicity}
	Assuming the family of cocycles $A_t$ is positively winding then the function $\rho:I\to\R$, $t\mapsto \rho(A_t)$, is   continuous and non decreasing.
\end{proposition}

\begin{proof}
	Consider a convergent sequence  $t_n\to t$  in $I$.
For each $n\in\N$ take   $\nu_{n}  \in \Prob(X\times\Pp^1)$ such that
$(\hat F_{t_n})_\ast \nu_{n}=\nu_{n}$ and $\pi_\ast \nu_{n} =\mu$. Denote by $\eta\in\Prob(X\times \Pp^1)$ any accumulation point of the sequence $\nu_{n}$.
We can easily check that
$(\hat F_{t})_\ast \eta=\eta$ and $\pi_\ast \eta=\mu$.
Since $\measuredangle(A_{t_n}(x)\hat v, \hat v)$ converges almost uniformly to
$\measuredangle(A_{t}(x)\hat v, \hat v)$ on $X\times\Pp^1$ 
and  by  item (3) of Proposition~\ref{Hn properties} the set of discontinuity points of $\measuredangle(A_{t}(x)\hat  v, \hat v)$   has  $\mu$-measure zero,
we have for some sub-sequence $n_i$,
\begin{align*}
	\lim_{i\to \infty} \rho(A_{t_{n_i}}) &= \lim_{i\to \infty} \frac{1}{\pi}\,\int \measuredangle(A_{t_{n_i}}(x)\hat v, \hat v)\, d\nu_{{n_i}}(x,\hat v) \\
	&= \frac{1}{\pi}\, \int \measuredangle(A_{t}(x)\hat  v, \hat v)\, d\eta(x,\hat v)=\rho(A_t).
\end{align*} 
Finally, the fact that this convergence holds for all sub-limits $\eta$ of $\nu_{n}$ implies that  \, $\lim_{n\to\infty}\rho(A_{t_n})$ \, does indeed exist and is equal to $\rho(A_t)$.

Given $t_0<t_0'$,  take a measurable set $\Omega\subset X$,
with full measure, i.e., $\mu(\Omega)=1$, such that for all $x\in \Omega$ and both $t\in \{t_0,t_0'\}$, 
$$ \rho(A_t)=\lim_{n\to \infty} \frac{1}{\pi n}\, \measuredangle(A^n_t(x)\, e_1, e_1) . $$
By the positive winding property
$$ \measuredangle(A^n_{t_0'}(x) e_1, e_1) -\measuredangle(A^n_{t_0}(x) e_1, e_1) =  \measuredangle(A^n_{t_0'}(x) e_1, A^n_{t_0}(x) e_1) \geq 0 $$
Hence taking limits
$$ \rho(A_{t_0'})=\lim_{n\to \infty} \frac{1}{\pi n}\,  \measuredangle(A^n_{t_0'}(x) e_1, e_1) 
\geq  \lim_{n\to \infty} \frac{1}{\pi n}\,  \measuredangle(A^n_{t_0}(x) e_1, e_1) =\rho(A_t) . $$
Therefore  $t\mapsto \rho(A_t)$ is a non-decreasing function.  
\end{proof}

\begin{proposition}
	Given any $t,t'\in I$, the relative rotation number
	$\rho(A_{t'})- \rho(A_t)$, $t'>t$ does not depend on the choice of the function $h$.\\
	In particular this relative rotation number  is intrinsically defined.
\end{proposition}

\begin{proof}
See~\cite[Remark A.8]{GK21}.
\end{proof}

\subsection{Homotopy properties}

In this section we discuss what happen with the length of  a the curve satisfying the winding property.

\begin{definition} Given  $I\ni t\mapsto A_t\in \GL_2(\R)$ smooth and a unit vector $v\in\Su^1$,  the length (oriented angle) of $A_t\, v/\norm{A_t\, v}\in\Su^1$ as $t$ ranges in $I$ is denoted by $\ell_I(A_t \, v)$.
This also agrees with the length of the projective curve $I\ni t\mapsto A_t\,\hat v$.
\end{definition}

\begin{proposition} Given a positively winding smooth curve $I\ni t\mapsto A_t\in \GL_2(\R)$  and $\hat v\in\Pp^1$, 
	$$ \ell_I(A_t \, \hat v) =\int_I \frac{ (A_t\, v) \wedge (\dot A_t\, v) }{\norm{A_t\, v}^2}\, dt . $$
\end{proposition}
\begin{proof}
	Follows from Proposition~\ref{derivative At v/norm}.
\end{proof}

Next lemma relates the asymptotic length of curves $J\ni t \mapsto A_t^n\, v$ with the Lebesgue-Stieltjes measure $d\rho$ determined by the  fibered rotation number $\rho(A_t)$.
\begin{lemma}\label{lem:lengthToRotation}
     For every $J \subset I$, every $\hat v\in \mathbb{P}^1$ and $\mu$-a.e. $x\in X$,
     \begin{align*}
         \lim_{n\to\infty}\frac{1}{n\pi}\ell_J(A^n_t(x)\, \hat v) = d\rho(J).
     \end{align*}
\end{lemma}
\begin{proof}
   By the winding property,  for any $\hat v\in \mathbb{P}^1$ and $t_0,\, t_1\in I$ with 
  $t_0 < t_1$ and  $x\in\Omega_{t_0}\cap \Omega_{t_1}$, 
    \begin{align*}
        \ell_{[t_0,\, t_1]}( A_t^n(x)\, \hat  v)
       % = \measuredangle(
       %     A^n_{t_1}(x)\,  v, A^n_{t_0}(x)\, v
       % )
        = \measuredangle(
            A^n_{t_1}(x)\, v,\, v
        ) - \measuredangle(
           A^n_{t_0}(x)\, v,\, v 
        ).
    \end{align*}
 Therefore this lemma follows from Proposition \ref{rotation number full measure set}.
\end{proof}

\begin{figure}[h]
	\begin{center}
		\includegraphics*[width=10cm]{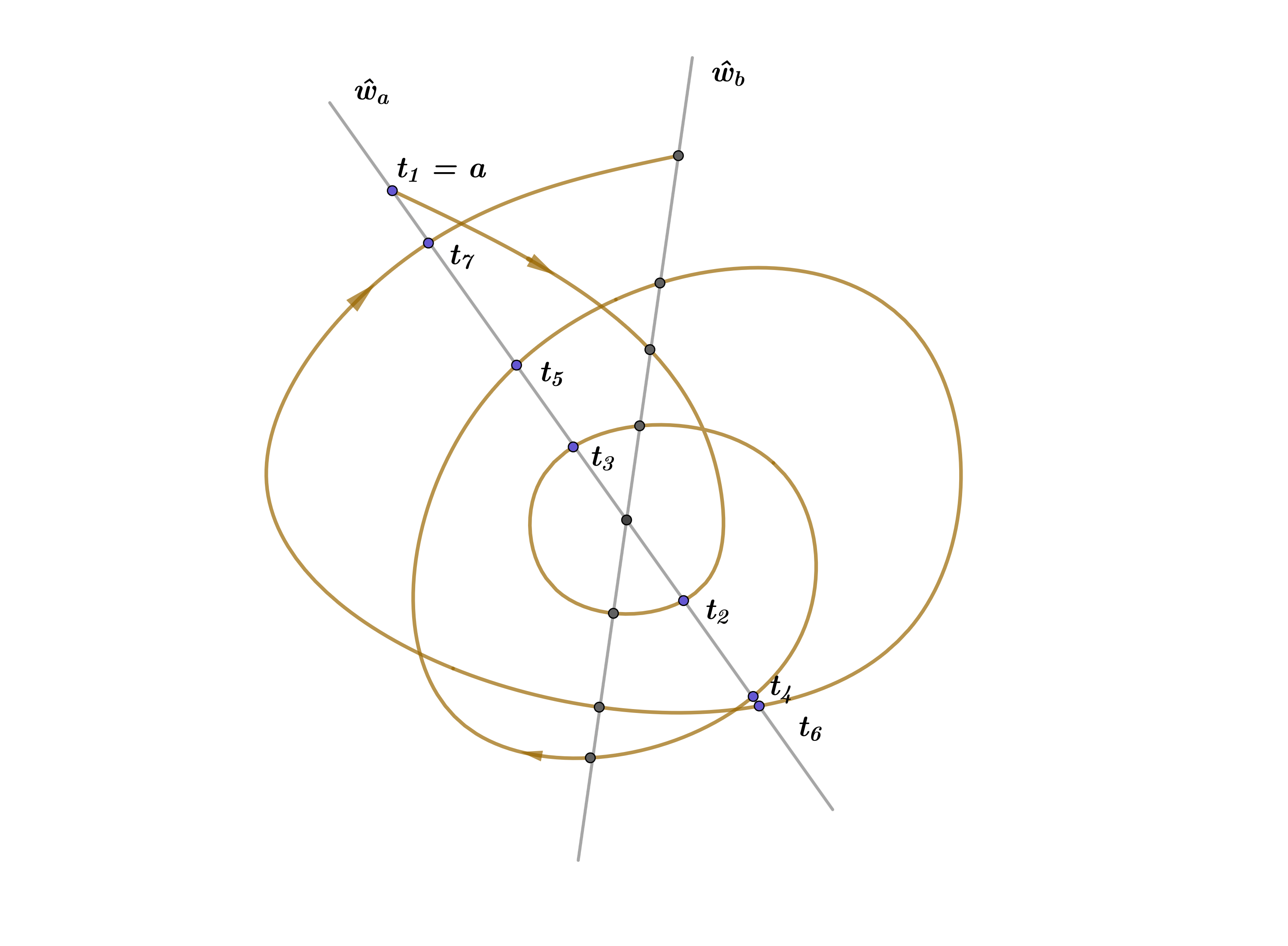}
	\end{center}
	\caption{Solutions of the equation $\hat A_t\, \hat v = \hat w_a := \hat A_a\, \hat v$, $t\in [a,\, b]$.}
	\label{solutions}
\end{figure}

\begin{lemma}
	\label{lemma 1 ell_I} 
	Given $v\in\Su^1$ and $I\ni t\mapsto A_t\in\GL_2(\R)$ a smooth curve positively winding, the following are equivalent:
	\begin{enumerate}
		\item $\ell_I(A_t\,v)\geq n\,\pi$;
		\item $\forall\, w\in \Su^1$, $\#\{t\in I\colon \hat A_t\, \hat v=\hat w \}\geq n$.
	\end{enumerate}
Moreover, if $I$ is closed, (1) or (2) hold, and
$\hat w\in\Pp^1$ is one of the end points of the curve $I\ni t\mapsto \hat A_t \, \hat v$, then the equation  equation $\hat A_t\, \hat v=\hat w$ has at least $n+1$ solutions in $t\in I$.
\end{lemma}

\begin{proof}
The direct implication (1) $\Rightarrow$ (2) is clear.
Assume (2), $I=[a,b]$ and set $w_t=A_t\, v$.
By (2) there are at least $n$ solutions
$a=t_1<t_2<\ldots < t_n\leq b$ of the equation 
$\hat A_t\, \hat v=\hat w_a $, see Figure~\ref{solutions}. For every small $\varepsilon>0$, notice that the equation $\hat A\, \hat v = \hat w_{t_n - \varepsilon}$ has $n-1$ solutions for $t\in [a,\, t_n]$. In fact, we have one solution of this equation in each interval $[t_i,\, t_{i+1}]$. Hence using the assumption, for each $\varepsilon$, there exists $s(\varepsilon) \in (t_n,\, b]$ which is the $n$-th solution of the equation $\hat A\, \hat v = \hat w_{t_n - \varepsilon}$. Making $\varepsilon\to 0$ and using the winding property we find another solution $t_{n+1}:=\lim_{s\to 0}s(\varepsilon)\in (t_n,\, b]$ of the equation $\hat A_t\, \hat v=\hat w_a$. This argument also shows that for any $c\in [t_{n+1},\, b]$, the equation $\hat A_t\, \hat v=\hat w_c $ has at least $n+1$ solutions in $[a,\, b]$. In particular, $l_I(A_t\, \hat{v})\geq n\pi$.
%The direct implication (1) $\Rightarrow$ (2) is clear.
%Assume (2), $I=[a,b]$ and set $w_t=A_t\, v$.
%By (2) there are at least $n$ solutions
%$a=t_1<t_2<\ldots < t_n\leq b$ of the equation 
%$\hat A_t\, \hat v=\hat w_a $. Assume by contradiction
%that there is no other solution, i.e., $\hat A_t\, \hat v\neq \hat w_a $ for $t_n<t\leq b$. 
%In particular $\hat w_b\neq \hat w_a$. Extend the curve $A_t$ to a small interval $[a,b+\delta[$   keeping the winding property. If $\delta$ is small
%$\hat w_s\neq \hat w_a$ for $s\in ]b,b+\delta[$, but for such $s
%$ the equation $\hat A_t\, \hat v=\hat w_s$
%can only have $n-1$ solutions in $t\in I$. This contradiction shows that $\hat A_t\, \hat v = \hat w_a $ has at least one more solution $t=t_{n+1}\in ]t_n, b]$. Hence  
%$$ \ell_I(A_t\, v)\geq \sum_{i=1}^n \ell_{[t_i,t_{i+1}]}(A_t\, v) = n\,\pi  $$
%which proves (1).
\end{proof}

\begin{corollary}
\label{corollary local winding}
	 Given  intervals $J\subset I$, $K\subset \Pp^1$ and $\hat v\in\Pp^1$, if for every $\hat w\in K$, the equation $A_t\, \hat v = \hat w$ has at least one solution for some $t\in J$, then $\ell_J(A_t\, \hat v)  \geq \mbox{length}(K)$.
\end{corollary}
\begin{proof}
 The curve $J\ni t\mapsto A_t\, \hat v$ winds positively and by assumption passes through all the $\hat w\in  K$. So, its length is larger or equal than $\geq \mbox{length}(K)$.

\end{proof}

\begin{definition}
	We say that a smooth curve $I\ni t\mapsto A_t\in\GL_2(\R)$ winds $n$ times around $\Pp^1$ if
	for all $\hat v\in\Pp^1$, $\ell_I(A_t\, \hat v)\geq n\,\pi$.
\end{definition}

\begin{corollary}
	\label{lemma winds n times}
	Given $I\ni t\mapsto A_t\in\GL_2(\R)$ a smooth curve positively winding, the following are equivalent:
	\begin{enumerate}
		\item  the curve $I\ni t\mapsto A_t\in\GL_2(\R)$ winds $n$ times around $\Pp^1$;
		\item $\forall\,  \hat  v, \hat w\in \Pp^1$,  $\#\{t\in I\colon \hat A_t\, \hat v=\hat w \}\geq n$.
	\end{enumerate}
\end{corollary}
\begin{proof}
Follows from Lemma~\ref{lemma 1 ell_I}.
\end{proof}

\begin{proposition}
\label{adding turns around P}
Given positively winding smooth curves
 $I\ni t\mapsto A_{i,t}\in\GL_2(\R)$,  $i=1,\ldots, k$, 
 if each $A_{i,t}$ winds $n_i\geq 0$ times around $\Pp^1$ then the composition curve
 $I\ni t \mapsto A_{n,t}\, \cdots \, A_{1,t}$
 winds $n_1+\cdots + n_k$ times around $\Pp^1$.
\end{proposition}

\begin{proof}
Using induction it is enough proving this statement for $k=2$. For notational simplicity we will denote the positively winding smooth matrix curves as $A_t$ and $B_t$. 
Assume $A_t$ winds $n$ times around $\Pp^1$ and 
$B_t$ winds $m$ times around $\Pp^1$.
We will prove that for any $\hat v\in\Pp^1$,
$\ell_I(A_t\, B_t\, v)\geq (n+m)\,\pi$.
By Corollary~\ref{lemma winds n times}, this will imply that $A_t\, B_t$ winds $n+m$ times around $\Pp^1$.

Let $I$ be an interval with end points $a<b$.
Take $\hat u=\hat B_a\, \hat v$ to be the starting point of the curve
$I\ni t\mapsto \hat B_t\, \hat v$ and let 
$\hat w=\hat A_a\, \hat u$ be the starting point of the curve
$I\ni t\mapsto \hat A_t\, \hat u$.
By the second statement in  Lemma~\ref{lemma 1 ell_I}, the equation
$\hat B_s\, \hat v=\hat u$ has at least $m+1$ solutions
$a=s_0<s_1 <\cdots < s_m\leq b$ in $I$, while
 the equation
$\hat A_t\, u=\hat w$ has also $n+1$ solutions
$a=t_0<t_1 <\cdots < t_n\leq b$ in $I$.
Let $\Gamma\colon I\times I\to \Pp^1$ be the continuous mapping
$\Gamma(t,s):=\hat A_t\, \hat B_s \, \hat v$.
The curve $\Gamma(t,t)$, with $t\in I$, is homotopic with fixed endpoints to the concatenation of the following three curves
\begin{itemize}
	\item $\Gamma_1\colon [a,t_n]\to \Pp_1$,
	$\Gamma_1(t):=\Gamma(t,a)=\hat A_t\, \hat B_a\, \hat v=\hat A_t \hat u$;
	\item $\Gamma_2\colon [a,b]\to \Pp_1$,
	$\Gamma_2(s):=\Gamma(t_n,s)=\hat A_{t_n}\, \hat B_s\, \hat v$;
	\item $\Gamma_3\colon [t_n, b]\to \Pp_1$,
$\Gamma_3(t):=\Gamma(t, b)=\hat A_{t}\, \hat B_b\, \hat v$.
\end{itemize}

\begin{figure}[h]
\begin{center}
	\includegraphics*[width=7cm]{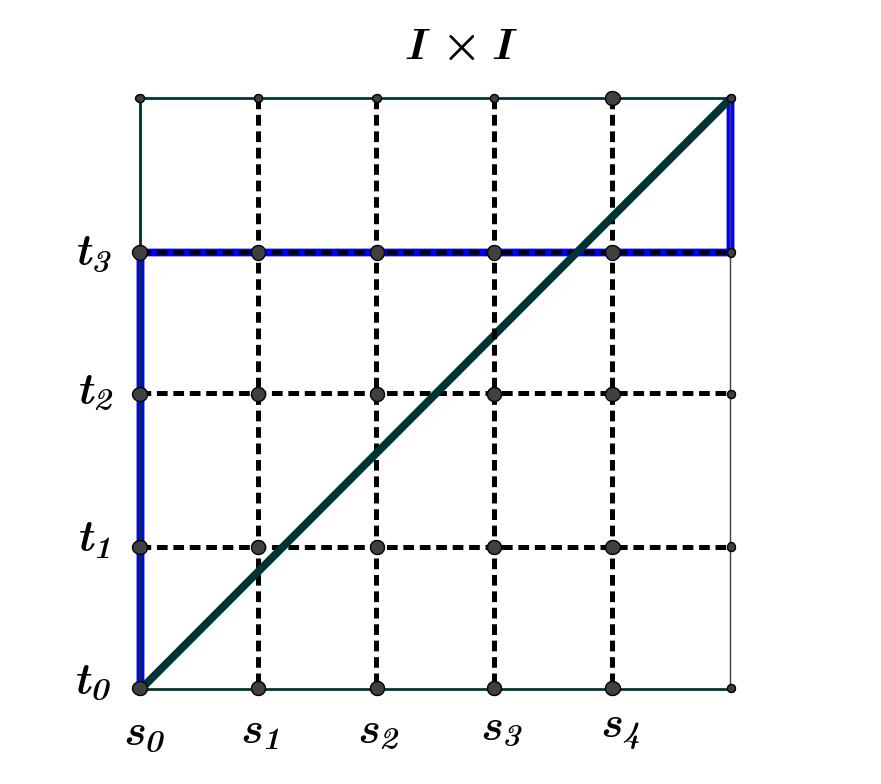}
\end{center}
\caption{The concatenation of $\Gamma_3\ast \Gamma_2\ast\Gamma_1$ is homotopic to the diagonal curve $\Gamma(t):=\hat A_t\, \hat B_t\, \hat v$.}
\label{homotopy pic}
\end{figure}
See Figure~\ref{homotopy pic}.
Since all these curves are positively winding and the diagonal $\Gamma(t,t)$ is homotopic to polygonal concatenation
$\Gamma_3\ast \Gamma_2\ast \Gamma_1$ through a homotopy which fixes the endpoints
\begin{align*}
	\ell_I(A_t\, B_t\, v)= \ell_I(\Gamma(t,t))&=
	\ell_{[a,t_n]}(\Gamma_1(t)) +
	\ell_{[a,b]}(\Gamma_2(s)) +
	\ell_{[t_n,b]}(\Gamma_3(t))\\
	&\geq n\, \pi + m\, \pi + 0 =(n+m)\,\pi .
\end{align*}
This concludes the proof.
\end{proof}

From now on we deal with affine families.

\begin{proposition}
	\label{affine winding}
	Given an affine curve $M_t:= A+t \, B\in\GL_2(\R)$ satisfying the winding property, for any $v\in\R^2$ such that $B\, v\neq 0$ we have \, $\ell_\R(M_t\, v)=\pi$.
\end{proposition}

\begin{proof}
	If $B\, v\neq 0$ then by the winding property $A_t v\wedge \dot{A}_tv>0$ for every $t\in \R$.  Notice that  $A\,v\neq 0$ because $A\in\GL_2(\R)$. Hence the straight-line
	$M_t \, v=A\, v + t\, B\, v$ induces a simple closed curve in $\Pp^1$ that begins and ends at $\proj{B}\, \hat v$. This implies that $\ell_\R(M_t\, v)=\length(\Pp^1)=\pi$.
\end{proof}

\begin{proposition}
	\label{polynomial winding}
	Let $M^n_t:= (A_n+t \, B_n)\,\cdots \, (A_1+t \, B_1)\, $, where
	$A_i\in\GL_2(\R)$ and $A_i+tB_i$ is positively winding  for  $i=1,\ldots, n$. Given $v\in\R^2\setminus\{0\}$, if $B_n B_{n-1}\, \cdots\, B_1\, v\neq 0$ \, then \,	$\ell_\R(M^n_t\, v)=n\,\pi$.
\end{proposition}

\begin{proof}
	The proof is made by induction over $n$. 
	
	The case $n=1$ follows by Proposition~\ref{affine winding}.
	
	Assume now that the induction hypothesis holds for a product of $n-1$ factors.  Then if $M^{n-1}_t:= (A_{n-1}+t \,B_{n-1})\,\cdots \, (A_1+t \, B_1)$, we have  \, $\ell_\R(M^{n-1}_t\, v)=(n-1)\,\pi$.
	For  $M^n_t:=(A_n+t\, B_n)\, M_t^{n-1}$  we define  $\Gamma\colon ]-\infty, +\infty[^2\to\Pp^1$ by
	$$\Gamma(t,s):= \widehat{ (A_n+t\,  B_n)}\,   \widehat{M_s^{n-1}}\hat v . $$ 
	Consider the product matrices  $B^{n-1}=B_{n-1}\, \cdots \, B_1$ and $B^n=B_n\,    B^{n-1}$. By assumption this matrix has either rank $1$ or  $2$.
	
	If $\rank(B_n)=2$ then $\Gamma$ extends continuously to  $\bar \Gamma\colon [-\infty, +\infty]^2\to \Pp$ and  the closed curve $\dot\R \ni t\mapsto \bar \Gamma(t,t)=\hat{M}^n_t\, \hat{v}\in\Pp^1$ is  homotopic to the concatenation of two closed curves,
	one of degree $1$,
	$\Gamma_1\colon \dot\R\ni t\mapsto \Gamma(t,-\infty)\in\Pp^1$, 
	and the other of degree $n-1$,
	$\Gamma_2\colon \dot \R\ni s\mapsto \Gamma(+\infty, s)$. Note that
	writing $C_n(t):=A_n+t\, B_n$,  we have
	$$\Gamma_1(t):= \Gamma(t,-\infty) =\hat C_n(t) \, \lim_{s\to-\infty} \hat M^{n-1}_s\, \hat v =\hat C_n(t)\hat B^{n-1}\, \hat v, $$
	the last equality holds because $B^{n-1}\,v\neq 0$.	 
	Likewise, since $B_n$ is invertible,
	$$\Gamma_2(s):=\Gamma(+\infty, s)=\lim_{t\to \infty} \hat C_n(t)\, \hat M^{n-1}_s\, \hat v =\hat B_n\, \hat M^{n-1}_s\, \hat v . $$
	By Proposition~\ref{affine winding}, $\deg(\Gamma_1)=1$, while
	by induction hypothesis we have $\deg(\Gamma_2)=n-1$.
	Hence by continuity of $\bar \Gamma$ in the square 	$[-\infty, +\infty]^2$,  
	$\hat M^n_t\, \hat v=\Gamma(t,t)$ is homotopic to
	the concatenation of $\Gamma_1(t)$ with $\Gamma_2(s)$
	which implies that $\deg_{t\in\dot\R}(\hat M_t\, \hat v)=1+(n-1)=n$.
	Thus we have $\ell_{\R}(M_t\, v)=n\,\pi$.

	Consider now the case $\rank(B_n)=1$.
	 By assumption, we have that  $\hat B^n\, \hat v = \lim_{s\to\pm\infty} \hat B_n\,\hat M^{n-1}_s\,\hat v$ is well-defined.
	By induction hypothesis the map $\dot{\R}\to \Pp^1$, $s\mapsto \hat{M}^{n-1}_s\hat{v}$, is a closed curve of  degree $n-1$. Hence, there are exactly $n-1$ elements $s^\ast_i\in ]-\infty,\infty[$ such that
	$B_n  \,  M^{n-1}_{s^\ast_i} \,  v=0$ for $i=1\ldots n-1$ and  we can extend $\Gamma(t,s)$ continuously to the set $[-\infty, +\infty]^2\setminus \{(\pm \infty,s_i^\ast): i=1\ldots n-1\}$. We claim that the closed curve $\dot\R \ni t\mapsto \Gamma(t,t)=\hat{M}^n_t\, \hat{v}\in\Pp^1$ is   homotopic to the concatenation of two closed curves,
	one of degree $1$ and the other of degree $n-1$.
	This implies  that $\dot \R \ni t\mapsto \Gamma(t,t)$ has degree $n$, and whence  $\ell_{\R}(M_t^n\, v)=n\,\pi$.

	The first of these curves,  $\Gamma_1 \colon \dot \R \to \mathbb{P}^1$,  is the same as above
	\begin{align*}
		\Gamma_1(t) &:= \Gamma(t,-\infty) = \hat{C}_n(t)\,\lim_{s\to -\infty} \hat{M}^{n-1}_s\hat{v} = \hat C_n(t)\, \hat B^{n-1}\,\hat v .
	\end{align*}
	It is a well defined and continuous curve with degree $1$. 
	
	The second curve $\Gamma_2\colon \dot \R\to \Pp^1$  cannot be 
	$\Gamma_2(s):=\Gamma(+\infty, s)$, because of the discontinuities at $s=s_i^\ast$.
	Fix a small number  $\varepsilon>0$ and choose $t^\ast\in]-\infty,+\infty[$ large enough so that the curve $\Gamma_1\vert_{[t^\ast,+\infty]}$ has length bounded by $\varepsilon$ and
	$\abs{s_i^\ast} <t^\ast$ for $i=1,\ldots, n-1$. Define $$\Gamma_{3}\colon[-\infty,-t^\ast]\to\mathbb{P}^1\text{ by }\Gamma_{3}(t):=\Gamma_1|_{[-\infty,-t^\ast]}(-t,-\infty)=\Gamma(-t,-\infty),$$  $$\Gamma_{4}\colon[-\infty,+\infty]\to\mathbb{P}^1\text{ by }\Gamma_{4}(s):=\Gamma(t^\ast,s)=\hat C_n(t^\ast)\, \hat M^{n-1}_s \hat v$$ and $$\Gamma_{5}\colon [t^\ast,+\infty]\to\mathbb{P}^1\text{ by }\Gamma_{5}(t):=\Gamma(t,+\infty).$$
	Finally let $\Gamma_2$ be the concatenation of $\Gamma_3$, $\Gamma_4$ and $\Gamma_5$  in this order.
	Let $\tilde{\Gamma}_i$ be the liftings of these curves for $i=1,\ldots, 5$. On one hand the   procedure for the case of $n=1$ implies that $\ell_{\mathbb{R}}(\tilde\Gamma_1)=\pi$. On the other hand, by the assumption on $t^\ast$ we get that $\ell_{[-\infty,-t^\ast]}(\tilde \Gamma_{3})<\varepsilon$ and   $\ell_{[t^\ast,+\infty]}(\tilde \Gamma_{5})<\varepsilon$, because they both match  the same arc of $\Gamma_1$. Finally, by the induction hypothesis, $\ell_{\mathbb{R}}(\tilde \Gamma_{4})=\ell_{\mathbb{R}}(M_s^{n-1}v)=(n-1)\pi
	$.
	Since $\Gamma(t,t)$ is homotopic to the concatenation of the curves  $\Gamma_1$, $\Gamma_3$, $\Gamma_4$ and $\Gamma_5$, we get  
	$$\pi-2\varepsilon+(n-1)\pi\leq \ell_{\R}(M^n_tv)\leq\pi+2\varepsilon+(n-1)\pi .$$
	Because the closed curve $\dot\R\ni t\mapsto \hat M_t^n\, \hat v\in\Pp^1$ does not depend on $\varepsilon$ and its length  is a multiple of $\pi$, we conclude that $\ell_{\R}(M^n_tv)=n\pi$.
\end{proof}

\subsection{Trace derivative}
In this subsection we establish some properties
about the derivatives of  the trace of
winding matrix curves.

\begin{proposition}
	\label{non zero trace prop}
	Let $I\ni t\mapsto M_t\in \SL_2(\R)$ be a  smooth curve with the  winding property.
	For any  $t\in\R$, if $\abs{ \tr(M_{t}) }<2$ then 
	$\frac{d}{dt}\left[ \,\tr (M_t)\, \right] \neq 0$.
\end{proposition}

\begin{proof}
Exclusively for the purpose of this proof we introduce the following non-oriented angle between non-collinear vectors $v,w\in\R^2$, defined by
$$ \measuredangle(v,w):=\arccos\left(\frac{v\cdot w}{\norm{v}\, \norm{w}} \right)\in ]0,\pi[, $$
 i.e., in terms of some Euclidean product in $\R^2$.

\begin{lemma}
	Given $A\in\SL_2(\R)$ elliptic, i.e., $\abs{\tr (A)}<2$,
	there exists a smooth measure $\mu_A\in \Prob(\Su^1)$ such that
	$$  \arccos\left(\frac{1}{2}\,\tr A \right) = \int_{\Su^1} \measuredangle(A\,v,v)\, d\mu_A(v) .$$
	Moreover this integral does not depend on  the Euclidean product  in $\R^2$.
\end{lemma}

\begin{proof}
	Write $A=M\, R_\alpha\, M^{-1}$ for some $M\in\SL_2(\R)$  and where 
	$R_\alpha$ is the angle $\alpha$ rotation with
	$\tr(A)=2\,\cos\alpha$.
	Let $\hat M:\Su^1\to \Su^1$ be the projective action
	induced by $M$ on $\Su^1$, i.e., $\hat M (v):=M\, v/\norm{M\, v}$. Next define $\mu_A=\hat M_\ast m$ where $m$ denotes the normalized Riemannian measure
	on $\Su^1$. Then $\hat A:\Su^1\to\Su^1$ is a circle homeomorphism which preserves the measure $\mu_A$, i.e.,  $\hat A_\ast\mu_A=\mu_A$.
	The angle $\measuredangle$ is a metric on $\Su^1$
	where the circle $\Su^1$ has diameter $\pi$ and length $2\pi$ but in general the rotation angle
	$\measuredangle(A\,v, v)$ is not constant for this metric.  Assume $\alpha$ is irrational$\pmod{2\pi}$.
	Then by the unique ergodicity of $\hat A$,
	$$\int_{\Su^1} \measuredangle(  A\, v,  v)\, d\mu_A(v) =\lim_{n\to\infty} \frac{1}{n}\,\sum_{i=0}^{n-1}
	\measuredangle(  A^{i+1} \, v,  A^i\, v) =\alpha $$
	measures the rotation number  of   $\hat A:\Su^1\to\Su^1$.
	For $\alpha$  rational$\pmod{2\pi}$ the result follows by continuity.
\end{proof}

If $\abs{ \tr (M_{t_0}) }<2$ then 
$M_t$ is elliptic for all $t$ in a small interval around $t_0$. Hence, for such $t$,  $\tr(M_t)=2\,\cos\theta(t)$ where 
$$ \theta(t)=\int_{\Su^1}\measuredangle(M_t\, v, v)\, d\mu_t(v)  $$
and $\mu_t\in\Prob(\Su^1)$ is the unique probability measure
invariant under the rotating  action of  $M_t$.
Because $M_t$ is elliptic, the oriented angle from $v$ to $M_t\, v$ is either always positive (for all $v$) or else
always negative.
By the positive winding assumption,
the curve
$M_t\, v/\norm{M_t\, v}$ rotates anti-clock wisely around the origin with positive speed. Hence, by ellipticity,
$$ \frac{d}{dt} \left[\, \measuredangle(M_t\, v, v)\, \right]_{t=t_0} \neq 0  $$
with a constant sign independent of the unit vector $v$.

Next choose an Euclidean product in $\R^2$ that makes
$M_{t_0}$ an orthogonal rotation and consider the associated angle $\measuredangle(\cdot,\cdot)$.
The map $v\mapsto \measuredangle(M_{t_0}\, v,v)$ is constant equal to $\theta(t_0)$, and because of this
$$  \theta'(t_0)= \int_{\Su^1}
\frac{d}{dt} \left[\, \measuredangle(M_t\, v, v)\, \right]_{t=t_0} \, d\mu_{t_0}(v) \neq 0 . $$

Therefore, since $0<\theta(t_0)<\pi$, one has $\sin\theta(t_0)>0$ and  
$$ \frac{d}{dt}\left[ \,\tr (M_t ) \, \right]_{t=t_0}
=2\, \sin \theta(t_0)\,  \theta'(t_0) \neq 0 ,  $$
which concludes the proof of Proposition~\ref{non zero trace prop}.
\end{proof}

\begin{definition}
A smooth function $f\colon\R\to\R$ is called \emph{log-concave}, respectively \emph{strictly log-concave}\,  if 
$$f(t)  f''(t)-(f'(t))^2\leq 0\quad \text{ resp.} \quad  f(t) f''(t)-(f'(t))^2< 0\quad  \forall \; t\in\R .$$
\end{definition}

\begin{remark}
\label{log-concave=>Morse function}
If all zeros of the function $f$ are isolated then $f$ is log-concave if and only if $\log \abs{f}$ is a concave function. Likewise, $f$ is strictly log-concave if and only if $\log \abs{f}$ is a strictly concave function. Notice that
$$ \frac{d^2}{dt^2} \log\abs{f(t)} = \frac{f(t)  f''(t)-(f'(t))^2}{f(t)^2 } . $$
A strictly log-concave function is always a Morse function, i.e., all its critical points are non-degenerate, and take strictly positive (negative) values at local maxima (minima).
\end{remark}

\begin{proposition}
\label{prop log-concave}
Consider a curve of the form 
$$ M^n_t=(A_n+t\, B_n) \, \cdots\, (A_2+t\, B_2) \, (A_1+t\, B_1)$$
where each factor $A_j+t\, B_j$ takes values in $\SL_2(\R)$
and is positively winding. Then
\begin{enumerate}
	\item Given $v, w\in\Su^1$, if  $\langle B_n\, \cdots \, B_1\, v, \, w \rangle \neq 0$ then $\R\ni t\mapsto \langle M_t^n\, v, \, w \rangle$ is a strictly  log-concave  polynomial of degree $n$ with $n$ simple real roots.
	\item If $B_n\, \cdots \, B_1\neq 0$ then  $\R\ni t\mapsto \tr(M^n_t)$ is a strictly  log-concave  polynomial of degree $n$  with $n$ simple real roots.
\end{enumerate}

\end{proposition}

\begin{proof}
Let $B^n:= B_n\, \cdots \, B_1$ and choose unit vectors $v,w\in\R^2$ such that $\langle B^n\, v, w\rangle$ $\neq 0$. Then the function
$f(t):= \langle M^n_t\, v, w \rangle$ 	
is strictly log-concave. Indeed, notice that $f(t)$ is a polynomial of degree $n$ with leading coefficient
$\langle B^n\, v, w\rangle\neq 0$.
By Proposition~\ref{polynomial winding},
$\ell_\R(M^n_t v)=n\,\pi$ and hence
the polynomial $f(t)=\langle M^n_t\, v, w \rangle$ must have $n$ distinct roots $t_1<t_2<\ldots < t_n$ which correspond to the values of the parameter $t\in\R$ where $M_t^n v$ crosses the line $w^\perp$. Hence
$f(t)=\langle B^n\, v, w\rangle\, \prod_{j=1}^n (t-t_j)$ and 
$$ \log \abs{f(t)} = \log\abs{\langle B^n\, v, w\rangle} + \sum_{j=1}^n \log \abs{t-t_j} $$
is strictly concave because it is a sum of the $n$ strictly concave functions $\log \abs{t-t_j}$. This shows that $f(t)$ is strictly log-concave. This concludes the proof of item (1).

To prove (2),  let  $\{e_1, e_2\}$ be an orthonormal basis  such that
$\langle B^n\, e_i, e_i\rangle\neq 0$ for $i=1,2$, and consider  the (half) trace function $f\colon\R\to\R$,
$$f(t):=\frac{\tr(M^n_t)}{2}=\frac{1}{2}\, \langle M_t^n\, e_1, e_1\rangle+\frac{1}{2}\, \langle M_t^n\, e_2, e_2\rangle =\frac{f_{1}(t)+ f_{2}(t)}{2}, $$
where each $f_i(t):=\langle M_t^n e_i,\, e_i\rangle$ 
is a strictly log-concave polynomial of degree $n$ by item (1),
for $i=1,2$.

Let $t_1<t_2<\ldots < t_n$ be the roots of $f_{1}(t)$ 
and $s_1<s_2<\ldots < s_n$ be the roots of $f_{2}(t)$. 
These roots are interlaced in the sense that
$$ \max\{t_{i-1}, s_{i-1}\}<\min\{t_{i}, s_{i}\}\quad \forall i=2,\ldots, n .$$
Otherwise we would have  
$t_{i-1}<t_i\leq s_{i-1}<s_i$ or 
$s_{i-1}\leq s_{i}<t_{i-1}\leq t_i$.
Keep in mind that the frame $\{M^n_t e_1, M^n_t e_2\}$ is moving anti-clockwisely while both its vectors maintain a positive orientation.
In the first case, as $t$ varies from $t=t_{i-1}$ to $t=t_i$ the vector $M^n_t e_1$ crosses twice the line $e_1^\perp$ while the second vector $M^n_t e_2$ is kept from crossing $e_2^\perp$. Hence 
$\measuredangle(M^n_{t_{i-1}} e_1, M^n_{t_{i-1}} e_1)>\pi$ 
while $\measuredangle(M^n_{t_{i-1}} e_2, M^n_{t_{i-1}} e_2)<\pi$. This is impossible because at some intermediate time the positive orientation of the frame would break. The second case is completely analogous.

\begin{figure}[h]
	\begin{center}
		\includegraphics*[width=12cm]{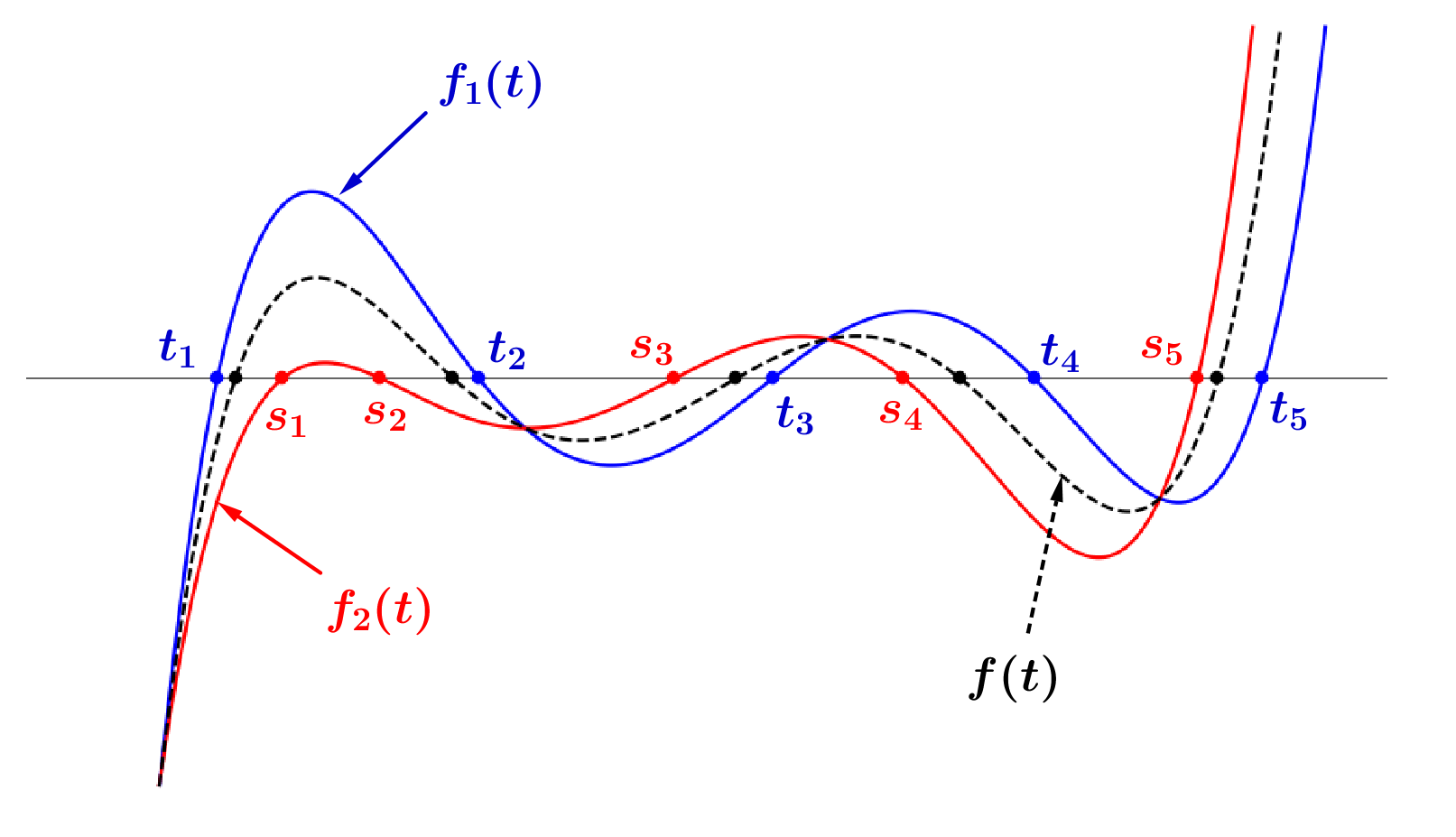}
	\end{center}
	\caption{The log concave functions $f_{1}(t)$,
	$f_{2}(t)$ and $f(t)$. }
	\label{interlace pic}
\end{figure}	
Finally, since $f(t)$ is a convex combination of $f_{1}(t)$ and $f_{2}(t)$, the function $f(t)$ has at least $n$ zeros, one between $t_i$ and $s_i$ for every $i=1,\ldots, n$. See Figure~\ref{interlace pic}.
Arguing as above we derive that being a polynomial of degree $n$, $f(t)$ must be strictly log-concave.
\end{proof}

By Remark~\ref{log-concave=>Morse function}
and Proposition~\ref{non zero trace prop},
the graph of $f(t):=\tr[ A_t^n(x) ]$ completely crosses $n$ times the  open horizontal strip $S:=\{(t,s)\in\R^2\colon -2<s<2 \}$, with local maxima  and  minima outside $S$.

\section{Proof of the main theorem}\label{mainproof}
In this section we %state the final results needed to 
prove Theorem~\ref{main}.

Assume  $T:X\to X$ is a homeomorphism on a compact metric space $X$
that preserves an ergodic measure $\mu\in\Prob(X)$ and let $A_t:X\to \GL_2(\R)$ be a family of cocycles of the form $A_t(x)=A(x)+tB(x)$ indexed in  $t\in \R$ and satisfying the assumptions 1-4.

By Assumption 4 there exists a continuous invariant decomposition
$\R^2=E_0(x)\oplus E_\infty(x)$,  where the  sub-bundle $E_0$ is associated with the top (finite) Lyapunov exponent and $E_\infty$ is associated with the second Lyapunov exponent, possibly $-\infty$.
The sub-bundles $E_0$ and $E_\infty$ determine   continuous functions  $\hat e_0:X\to \Pp^1$ and $\hat e_\infty:X\to \Pp^1$ respectively.
Consider also the adjoint cocycle
$B^\ast(x):=B(T^{-1} x)^t$
over the base map $T^{-1}\colon X\to X$, which shares with $B$ the same Lyapunov exponents
$L_1(B^\ast)=L_1(B)>L_2(B)=L_2(B^\ast)$. 
Let $\hat e_0^\ast:X\to \Pp^1$ and $\hat e_\infty^\ast:X\to \Pp^1$ denote the corresponding continuous functions associated with its dominated splitting decomposition.

\begin{lemma}
\label{L1(B) limit}
For every $x\in X$,
if $v\notin \hat e_\infty(x)$ and $w\notin \hat e_\infty^\ast(x)$ then 
$$L_1(B)=\lim_{n\to\infty} \frac{1}{n}\, \log \abs{ \langle B^n(x)\, v, w \rangle} .$$ 
\end{lemma}

\begin{proof}
We relate the  directions $\hat e_0(x)$, $\hat e_\infty(x)$, $\hat e_0^\ast(x)$ and $\hat e_\infty^\ast(x)$  with the singular vectors of the matrices $B^n(x)$ and $(B^\ast)^n(x)$.
See the definitions in Subsection~\ref{appendix:projective analysis}.

To simplify notations we set
\begin{align*}
	\medir_n(x)  := \medir(B^n(x)) \quad \text{ and }\quad 
	\ledir_n(x) := \ledir(B^n(x)) 
\end{align*}
as well as
\begin{align*}
	\medir_n^\ast(x) & = \medir^\ast(B^n(T^{-n} x))
	= \medir\left(  (B^{\ast})^n(x) \right)  \\
	\ledir_n^\ast(x) & = \ledir^\ast(B^n(T^{-n} x)) = \ledir\left(  (B^{\ast})^n(x) \right) \\
\end{align*}
Since $L_1(B)>L_2(B)$, all these four sequences converge  $\mu$-almost surely respectively to
$\hat e_0^\ast(x)$, $\hat e_\infty(x)$,
$\hat e_0(x)$ and $\hat e_\infty^\ast(x)$.
Moreover  
$$\hat e_0^\ast(x)=\hat e_\infty(x)^\perp\quad \text{ and }\quad \hat e_\infty^\ast(x)=\hat e_0(x)^\perp .$$
 See Chapter 4 of \cite{LLE}.
Given unit vectors $v, w\in \Su^1$,
$$v= \langle v, \medir_n(x) \rangle\,\medir_n(x) +
\langle v, \ledir_n(x) \rangle\,\ledir_n(x),  $$
and whence
$$B^n(x)\,v = \norm{B^n(x)}\,\langle v, \medir_n(x) \rangle\,\medir_n^\ast(x) +
\conorm(B^n(x))\, \langle v, \ledir_n(x) \rangle\,\ledir_n^\ast(x)  $$
which implies that  $\langle B^n(x)\,v, w\rangle$ is equal to 
$$\langle v, \medir_n(x) \rangle\, \langle w, \medir_n^\ast (x) \rangle\, \norm{B^n(x)} +
\langle v, \ledir_n(x) \rangle\, \langle w, \ledir_n^\ast (x) \rangle\, \conorm(B^n(x))  . $$
Note that as $n$ large, the first term dominates.

Since  $v\neq \hat e_\infty(x)=\hat e_0^\ast(x)^\perp$,  
$$ \lim_{n\to +\infty}\frac{1}{n}\,\log \abs{
	\langle v, \medir_n(x) \rangle } =
\lim_{n\to +\infty}\frac{1}{n}\,\log \abs{ \underbrace{
		\langle v, \hat e_0^\ast(x) \rangle }_{\neq 0}  }  = 0  .$$
Analogously, since  $w\neq \hat e_\infty^\ast(x)=\hat e_0(x)^\perp$,
$$ \lim_{n\to +\infty}\frac{1}{n}\,\log \abs{
	\langle w, \medir_n^\ast(x) \rangle } =
\lim_{n\to +\infty}\frac{1}{n}\,\log \abs{
	\underbrace{\langle w, \hat e_0(x) \rangle}_{\neq 0} }  = 0 . $$

Taking absolute values, logarithms, dividing by $n$, and the limit as $n\to +\infty$, we have  $\mu$-almost surely
\begin{align*}
\lim_{n\to\infty} \frac{1}{n}\, \log \abs{ \langle B^n(x)\, v, w \rangle} &= \lim_{n\to\infty} \frac{1}{n}\, \log \norm{ B^n(x)} = L_1(B)  ,
\end{align*}
which concludes the proof.
\end{proof}

\begin{lemma}
\label{bad winding directions}
For every $x\in X$ there exists a countable set of (bad) directions $\Bscr_x\subset \Su^1$ such that for any $v \in \Su^1\setminus \Bscr_x$, $w\in\Su^1$ and $n\in\N$ the function
$f(t):=\langle A_t^n(x)\, v, w \rangle$ is a polynomial of degree $n$.
\end{lemma}

\begin{proof}
	Consider $E:X\to \Mat_2(X)$ defined by
	$E(x):=A(x)^{-1}\, B(x)$. The winding assumption implies that  $E(x)$ has either complex, non real eigenvalues or else it has a single real eigenvalue with multiplicity two, for all $x\in X$. See Propositions~\ref{pos winding charact} and \ref{pos winding charact, case Delta=0}  in  Subsection~\ref{subsection: Conditions to gurantee 1-3} of the Appendix.

	Given $x\in X$ consider the countable set,
	$$ \Bscr_{x}:=\left\{v\in \Su^1  \colon \text{for some } j\geq 0,\, A^j(x)\, v\,  \text{ is an eigenvector of}\, E(T^j x) \,\right\} .$$
	If $v\in   \Bscr_x^\complement$ then the smooth curve $A_t(T^j x)=A(T^j x)+t B(T^j x)$ is positively winding 
	(for all $j\geq 0$) and whence by Proposition~\ref{prods of pos winding are pos winding}, see also Proposition~\ref{polynomial winding}, the curve
	$A_t^n(x)$ is positively winding and $f(t):=\langle A_t^n(x)\, v, w \rangle$ is a polynomial of  degree $n$.
\end{proof}

Consider the Lebesgue-Stieltjes measure $d\rho$ associated with the continuous   non-decreasing function (see Proposition~\ref{rho continuity and monotonicity})  $\rho(t):=\rho(A_t, A_{t_0} )$.

\begin{lemma}
	\label{Johnson trivial half}
	If the cocycle $A_t$ has dominated splitting over some open interval $I\subseteq \R$ then	the function $\rho(t)$ is constant on $I$.
\end{lemma}

\begin{proof}
	Let $I\subset \R$ be an open interval.
	Assume that  $A_t$ has dominated splitting for all $t\in I$ and consider the unstable or dominating direction $\hat e_{u}(x,t)\in\Pp^1$ of   $A_t$.
	By splitting domination, the map $\hat e_u\colon X\times I\to \Pp^1$ is continuous. Given $t>s$ in $I$, $x\in\Omega_t\cap\Omega_s$ (see Proposition~\ref{rotation number full measure set} for the definition of the sets $\Omega_t$ and $\Omega_s$) and some appropriate $\hat v\in\mathbb{P}^1$,	
	\begin{align*}
	\rho(A_t)-\rho(A_s) &=\lim_{n\to\infty} \frac{1}{\pi n}\, \measuredangle(A_t^n(x)\, v, v) - \frac{1}{\pi n}\, \measuredangle(A_s^n(x)\, v, v) \\
	&=\lim_{n\to\infty} \frac{1}{\pi n}\, \measuredangle(A_t^n(x)\, v,  A_s^n(x)\, v) \\
	&=\lim_{n\to\infty} \frac{1}{\pi n}\, \measuredangle(\hat e_{u}(T^n x,t), \hat e_u(T^n x,s)) =0 ,
	\end{align*} 
because the angle  $\measuredangle(\hat e_{u}(x,t), \hat e_u(x,s))$ is uniformly  bounded,
	which proves that $\rho$ is constant over $I$.
\end{proof}

\begin{lemma}
	\label{compact support}
The cocycle $A_t$ has  dominated splitting  outside some compact interval $[-a,a]$. In particular, the support of the measure  $d\rho$ is compact.
\end{lemma}

\begin{proof}
Any continuous family of cones $\{C_x\}_{x\in X}$
adapted to the cocycle $B$, which has dominated splitting,
is shared by the cocycles $t\, B$ and $A_t=A+t\, B$ for all 
$t$ with large enough absolute value, i.e., $|t|>a$.
It follows that $A_t$ has dominated splitting for
all  $|t|>a$ and whence by Lemma~\ref{Johnson trivial half},
$\supp(d\rho)\subset [-a,a]$.	
\end{proof}

%
%\begin{figure}[h]
%	\begin{center}
%		\includegraphics*[width=12cm]{windnumbr.png}
%	\end{center}
%	\caption{Interlaced polynomial roots of $\langle A^n_t e_k, e_i \rangle=0$ and  $\langle A^n_t e_k, e_j \rangle=0$  }
%	\label{winding curve}
%\end{figure}

\bigskip

\begin{proof}[Proof of Theorem~\ref{main}]
By Lemma~\ref{compact support} the non-decreasing function $\rho(t)=\rho(A_t)$ has compact support
contained in some interval $[-a,a]$.

Fixing $t\in\C$, for $\mu$-almost every $x\in \Omega$ (depending on $t$)
\begin{align*}
	L_1(A_t) &=\lim_{n\to\infty} \frac{1}{n}\,\log \norm{A^n_t(x)} \\
	&\quad =\lim_{n\to\infty} \frac{1}{n}\,\log \max_{i,j=1,2}\abs{\langle A^n_t(x)\, e_i, e_j \rangle}\\
	&\quad =\max_{i,j=1,2} \left[ \lim_{n\to\infty} \frac{1}{n}\,\log \abs{\langle A^n_t(x)\, e_i, e_j \rangle}\right],
\end{align*}
where $\{e_1, e_2\}$ is any basis of $\R^2$.

On the other hand we are going to prove that for $\mu$-almost $x\in X$, taking an appropriate basis $\{e_1, e_2\}$ of $\R^2$ (depending on $x$), for all $i,j=1,2$ and $t\in \C\setminus \R$,
$$  \lim_{n\to\infty} \frac{1}{n}\,\log \abs{\langle A^n_t(x)\, e_i, e_j \rangle} = L_1(B)+\int \log |t-s|\, d\rho(s) .$$
Hence the identity~\eqref{Thouless ID} holds for all $t\in\C\setminus \R$.

Indeed this enough by the following argument.
The function $t\mapsto L_1(A_t)$ is upper semi-continuous
and since  
$$u_{n}\colon \C\to\R, \quad u_{n}(t):=  \frac{1}{n}\,\int_X \log \norm{A^n_t(x)}\, d\mu(x) $$
 is a family of  subharmonic functions,
uniformly bounded  from above. Note that  $\{u_{2^j}\}_{j\geq1}$ is a convergent and decreasing subsequence of  subharmonic functions, so the limit function $L_1(A_t)=\lim_{n\to+\infty} u_{n}(t)$ is a subharmonic function. The right-hand-side in~\eqref{Thouless ID} is also a subharmonic function since it is an average of the subharmonic functions   $v_s(t):=\log |t-s|$.
Finally, because these two subharmonic agree Lebesgue almost everywhere, they must coincide everywhere,
 see~\cite[Theorem 1.1]{CraigSimon}. This proves that~\eqref{Thouless ID}
holds for all $t\in\C$.

To finish the proof of Theorem~\ref{main} we establish the previous  claim. 

Consider the full measure set $\Omega=\cap_{\beta\in\Q} \Omega_\beta \subset X$, where each $\Omega_\beta$ is the full measure set  in
Proposition~\ref{rotation number full measure set} associated with 
$\beta\in\Q$. Taking  $x\in\Omega$, 
 and an orthonormal basis $\{e_1,e_2\}$ of $\R^2$ consisting of vectors which do  not match the dominated directions $\hat e_\infty(x)$ and $\hat e_\infty^\ast(x)$ of the cocycle  $B$ and its adjoint $B^\ast$ at $x$,
 by Lemma~\ref{L1(B) limit}
 we have
$$L_1(B)=\lim_{n\to\infty} \frac{1}{n}\, \log \abs{ \langle B^n(x)\, e_i, e_j \rangle} \quad \forall\, i,j=1,2.$$ 
We can also take this basis $\{e_1,e_2\}$  outside the countable set $\Bscr_x$ of Lemma~\ref{bad winding directions}
so that the functions
$f_{i,j}(t):= \langle A_t^n(x)\, e_i, e_j \rangle$ are 
polynomials of degree $n$ by item (1) of Proposition~\ref{prop log-concave}. By the same item, each of these functions has exactly $n$ roots, denoted by
$t_1(i,j)<t_2(i,j)<\cdots <t_n(i,j)$.
Hence $\langle B^n(x)\, e_i, e_j \rangle$ is the leading coefficient of $f_{i,j}(t)$
and we have

\begin{align*}
	\lim_{n\to\infty} \frac{1}{n}\,\log \abs{f_{i,j}(t)}  &=
	\lim_{n\to\infty} \frac{1}{n}\,\log \abs{\langle B^n(x)\, e_i, e_j \rangle\, \prod_{k=1}^n (t-t_k(i,j))}\\
	&=
	\lim_{n\to\infty} \left[\, \frac{1}{n}\,\log \abs{\langle B^n(x)\, e_i, e_j \rangle}+\frac{1}{n}\,\sum_{k=1}^n \log\abs{t-t_k(i,j)}\, \right]  \\
	&=
	L_1(B) +\lim_{n\to\infty}  \frac{1}{n}\,\sum_{k=1}^n \log\abs{t-t_k(i,j)}  .
\end{align*}
Therefore, it is enough proving now that  
\begin{equation}
\label{lim sum log |t-tk|}
\lim_{n\to \infty}\frac{1}{n}\, \sum_{k=1}^n \log |t-t_k| = \int_\R \log |t- s|\, d\rho(s) \quad \forall\, t\in\C\setminus\R  
\end{equation} 
for which we need the following.

\begin{proposition}
	\label{lim sum f(tk)=int f drho} 
	For any $x\in\Omega$, $\{e_1, e_2\}$ as above, $i,j=1,2$ and any continuous function $\varphi:\R\to\R$,
	if $t_k=t_k(i,j)$, $k=1,\ldots, n$ are  the $n$ roots of the polynomial equation
	$\langle A_t^n (x)\, e_i, e_j\rangle=0$  then  
	$$ \lim_{n\to \infty}\frac{1}{n}\, \sum_{k=1}^n \varphi(t_k) = \int_\R \varphi(s)\, d\rho(s) .$$
\end{proposition}

\begin{proof}
	Fix $\varepsilon>0$ and let us prove that for all large enough $n$,
	$$\left| \frac{1}{n}\, \sum_{k=1}^n \varphi(t_k) - \int_\R \varphi(s)\, d\rho(s) \right| <2\, \varepsilon .$$
	
	The measure $d\rho$ is supported on a compact interval $[-a,a]$.
	Because $\varphi$ is uniformly continuous on $[-a,a]$ there exists $\delta>0$ such that for any decomposition
	$-a=\beta_0<\beta_1 <\cdots <\beta_{m-1}<\beta_m=a$ of the interval $[-a,a]$ with diameter $\max_{1\leq l\leq m} \abs{\beta_l-\beta_{l-1}}<\delta$,
	any Riemannian sum
	$$ \sum_{l=1}^m \varphi(s_l)\, (\rho(\beta_l)-\rho(\beta_{l-1})) $$
	with $s_l\in [\beta_{l-1}, \beta_l]$ for all $l=1,\ldots, m$, satisfies
	$$ \biggr\lvert \sum_{l=1}^m \varphi(s_l)\, (\rho(\beta_l)-\rho(\beta_{l-1})) - \int_\R \varphi\, d\rho\bigr \rvert <\varepsilon .$$
	
	Fix $\delta>0$ and the decomposition
	$-a=\beta_0<\beta_1 <\cdots <\beta_{n-1}<\beta_m=a$ with
	$\beta_l\in\Q$, for $l=1,\ldots, m$, and  diameter less than $\delta$ as above.
	Let $L :=\max\{\abs{\varphi(x)}\colon x\in [-a,a]\}$
	and take $0<\eta<\varepsilon/(2\,m\, L)$. Since $x\in \Omega\subset \cap_{l=1}^m \Omega_{\beta_l}$, by Lemma \ref{lem:lengthToRotation} there exists $n_0\in\N$ such that $2n_0^{-1}<\varepsilon/(2\, m\, L)$ and for all $n\geq n_0$,
	$$  \left| d\rho([\beta_{l-1},\beta_l[) - \frac{1}{\pi n}\, \measuredangle( A_{\beta_l}^n(x)\, v, A_{\beta_{l-1}}^n(x)\, v  ) \right| <\eta\quad \forall l=1,\ldots, m .$$
	Since from $t_{l-1}$ to $t_l$ the curve $t\mapsto A_t^n(x)v$ gives one turn around the projective space, i.e. $\measuredangle( A_{t_l}^n(x)v, A_{t_{l-1}}^n(x)v) = \pi$, setting
	$$ N_l:= \# \{ k\in \{1,\ldots, n\} \colon t_k \in [\beta_{l-1},\beta_l[ \, \}\}, \qquad l=1,\ldots, m, $$
	we have that 
	$\abs{ \pi\, (N_l-1)  - \measuredangle(A_{\beta_l}^n(x)\, v, A_{\beta_{l-1}}^n(x)\, v  )} <2\pi$, which implies that
	$$  \left| d\rho([\beta_{l-1},\beta_l[ ) - \frac{N_l-1}{n}  \right| <\eta+\frac{2}{n_0} \quad \forall l=1,\ldots, m .$$
	Hence
	$$ \frac{1}{n}\, \sum_{k=1}^n \varphi(t_k) =
	\sum_{l=1}^m \frac{N_l-1}{n}\, \left(\frac{1}{N_l-1}\, \sum_{  t_k\in [\beta_{l-1},\beta_l[} \varphi(t_k)\right)$$
	where by continuity of $\varphi$ we can find, for each $l=1,\ldots, m$, $t_l^\ast\in [\beta_{l-1},\beta_l[$ such that
	$$\varphi(t_l^\ast) = \frac{1}{N_l-1}\, \sum_{  t_k\in [\beta_{l-1},\beta_l[} \varphi(t_k) $$ and
	$$\frac{1}{n}\, \sum_{k=1}^n \varphi(t_k)=\sum_{l=1}^m \frac{N_l-1}{n}\,\varphi(t_l^\ast).$$
	Therefore
	$$\left|  \frac{1}{n}\, \sum_{k=1}^n \varphi(t_k) - \sum_{l=1}^m (\rho(\beta_l)-\rho(\beta_{l-1}))\, \varphi(t_l^\ast)   \right|$$
	is bounded by
	\begin{align*}
		\sum_{l=1}^m \left|  \frac{N_l-1}{n} -   d\rho([ \beta_{l-1},\beta_l [ )\right| \,  \left|\varphi(t_l^\ast)   \right|
		&\leq \left(\eta+\frac{2}{n_0}\right)\, m\, L <\varepsilon
	\end{align*}
	and then
	\begin{align*}
		\left|  \frac{1}{n}\, \sum_{k=1}^n \varphi(t_k) - \int \varphi(s)\, d\rho(s)  \right| &\leq
		\left|  \frac{1}{n}\, \sum_{k=1}^n \varphi(t_k) - \sum_{l=1}^m (\rho(\beta_l)-\rho(\beta_{l-1}))\, \varphi(t_l^\ast)   \right| \\
		&\qquad + \left|  \sum_{l=1}^m (\rho(\beta_l)-\rho(\beta_{l-1}))\, \varphi(t_l^\ast)  - \int \varphi(s)\, d\rho(s) \right|\\
		&< \varepsilon + \varepsilon =2\, \varepsilon.
	\end{align*}
	which establishes the wanted convergence.
\end{proof}

%
%\begin{corollary}
%	For any $x\in\Omega$, $\{e_1, e_2\}$ as above, $i,j=1,2$, and  $t\in\C$, if $t_k=t_k(i,j)$, $k=1,\ldots, n$, are  the $n$ roots of 
%	$\langle A_t^n (x)\, e_i, e_j\rangle=0$  then    
%$$ \lim_{n\to \infty}\frac{1}{n}\, \sum_{k=1}^n \log \vert t - t_k \vert  = \int_\R \log \vert t - s\vert \, d\rho(s)  $$
%is a subharmonic function.
%\end{corollary}

%By Proposition 8, we have 
%$$ 
%L_1(A_t)= \lim_{n\to +\infty} \frac{1}{n}\,\log \abs{\langle A_t^n(x)\, v, v\rangle}$$
%holds for any $t\in I$, any non-zero $v\in \R^2$ and $\mu$ a.e. $x\in X$. Observe that the dependence of $A_t$ on $t$ is analytic, so by an alternative definition of the Lyapunov exponent (consider sequence $2^j$ such that the corresponding finite Lyapunov exponent is a sequence of decreasing subharmonic functions), we have $L_1(A_t)$ is subharmonic in $t$ in some complex band $D$ of $I$.
%Let $t_i$ be the same in Proposition 25. Since $L_1(B)$ is finite, by direct computation we have
%\begin{align*}
%	\lim_{n\to +\infty} \frac{1}{n}\log \abs{\langle A_t^n(x)\, v, v\rangle}
%	&=\lim_{n\to +\infty} \frac{1}{n}\log\lVert B^n(x)\rVert+\lim_{n\to +\infty} \frac{1}{n}\sum_{i=1}^{n}\log\abs{t-t_i}\\
%	&=L_1(B)+\lim_{n\to +\infty} \frac{1}{n}\sum_{i=1}^{n}\log\abs{t-t_i}
%\end{align*}
%In particular, we obtain that
%$$\lim_{n\to +\infty} \frac{1}{n}\sum_{i=1}^{n}\log\abs{t-t_i}$$ exists and is also subharmonic in $t$.

Given $\delta>0$ and $t\in\R$, consider the  family  $\varphi_{t,\delta}\colon \R\to\R$ of continuous functions defined by $$\varphi_{t,\delta}(s):=\log\abs{t+i\delta-s}.$$ 
By Proposition~\ref{lim sum f(tk)=int f drho}  we have
for all $t\in\R$ and $\delta>0$,
$$
\lim_{n\to +\infty} \frac{1}{n}\sum_{k=1}^{n}\log\abs{t+i\delta-t_k}=\int_\R \log\abs{t+i\delta-s}d\rho(s).
$$
Therefore~\eqref{lim sum log |t-tk|} holds for all $t\in\C\setminus\R$.
%
%In other words 
%$$
%\lim_{n\to +\infty} \frac{1}{n}\sum_{k=1}^{n}\log\abs{t-t_k}=\int_\R \log\abs{t-s}\, d\rho(s).
%$$
%holds for all $t\in\C\setminus\R$.
%
%
% Interlacing defintion
%\begin{definition}
%	Let $f(t)$ and $g(t)$ be two polynomial functions  of degree $n$ with roots 
%	$t_1<t_2<\cdots < t_n$ and $s_1<s_2<\cdots < s_n$, respectively. We say that the roots of $f(t)$ and $g(t)$  are \emph{interlaced} if 
%	either 
%	$$ t_1<s_1<t_2<s_2 < \cdots <  t_n < s_n  $$
%	or else 
%	$$ s_1<t_1<s_2<t_2 < \cdots <  s_n < t_n  .$$
%\end{definition}
%
%
%
%\begin{lemma}
%For any $k=1,2$, the following pairs of polynomial functions have interlaced roots:
%\begin{enumerate}
%	\item $f_{k,1}(t):= \langle A_t^n(x)\, e_k, e_1 \rangle$  and $f_{k,2}(t):= \langle A_t^n(x)\, e_k, e_2 \rangle$;
%	\item $f_{1,k}(t):= \langle A_t^n(x)\, e_k, e_1 \rangle$  and $f_{2,k}(t):= \langle A_t^n(x)\, e_k, e_2 \rangle$,
%\end{enumerate}
%\end{lemma}
%
%
%\begin{proof}
%	
%See Figure~\ref{winding curve}.
%\end{proof}
%Then for every $\varphi\in C^0([-a,a],\mathbb{R})$ taking $\delta>0$ and a sequence $\beta_i\in[-a,a]$ for $i=1\ldots m$ with $-a=\beta_0<\beta_1<\ldots <\beta_m=a$ such that $\sup_i|\beta_i-\beta_{i-1}|<\delta$ the limit 
%$$\lim_{\delta\to 0} \sum_{i=1}^m \varphi(s_i)\left(\rho(\beta_i)-\rho(\beta_{i-1})\right)=\int \varphi(t)d\rho$$
%where $s_i\in [\beta_{i-1},\beta_i[$. 
%
\end{proof}

\section{Affine families of \texorpdfstring{$\GL_2$}{TEXT}-cocycles}
\label{affine families}

In this section we provide some easily verifiable sufficient conditions for the Assumptions 1, 2 and 4.
We assume that $A\colon X\to \GL_2(\R)$ and $B\colon X\to \Mat_2(\R)$ are continuous functions such that
$$ A_t:=A+t\, B = A\, (I+t\, E) \qquad \forall\, t\in \R  $$
where  $E:X\to \Mat_2(\R)$ denotes the function $E:=A^{-1}\, B$.

\subsection{Conditions for invertibility}

To ensure the invertibility of $A_t$ we have the following criterion.

\begin{proposition} 
\label{prop invertibility criterion}
If there exists a constant $r>0$ such that
$$E(x)^2=0 \; \text{ or  }\; \Delta_{E(x)}:=4\, \det E(x) -  (\tr \,E(x))^2\geq  r>0,\quad \forall x\in X $$
	 then the cocycle $A_t$ satisfies Assumption 1, i.e.,
	there exist positive constants $c$ and $R$ such that
	$|\det A_t(x)|\geq c>0$ for all $(x,t)\in X\times\mathscr{S}_R$.
\end{proposition}

\begin{proof}
	The following argument relies on the conclusions of Lemma~\ref{lemma: det(I+t E) formula}. 
    If $E^2=0$ then $\tr E=\tr E^2=\det E = 0$ and $\det(I+t\, E)\equiv 1$, so that there is nothing to prove. Otherwise $\Delta_E\geq  r>0$ which implies that $\det(E) \geq r/4>0$.
    
    Taking $\ell:=\max\{ \det E(x)\colon x\in X\}$ and $\delta>0$ small enough so that $4\,\ell^2\,\delta^2< r$ and using Lemma \ref{lemma: det(I+t E) formula} we have for $|\Imp t|<\delta$
    \begin{align*}
        \abs{\det(I+t\, E)}\geq |\det E|\,\left[ -\delta^2+ \frac{\Delta_E}{4\,(\det E)^2}\right] \geq \frac{r}{4}\, \left( \frac{r}{4\,\ell^2} -\delta^2\right) >0.
    \end{align*}
    This concludes the proof with $R=\delta$ and 
    $c=\frac{r}{4}\, \left( \frac{r}{4\,\ell^2} -\delta^2\right)$.
\end{proof}

For Assumption 1 to hold with  $R=\infty$, the polynomials $\det(A_t(x))$ must be constant. A special case of interest is the following:

\begin{proposition}
	\label{lemma: SL2 criterion}
    For all $x\in X$, the following are equivalent
    \begin{enumerate}
        \item $A_t(x)\in \SL_2(\C)$ for all $t\in \C$;
        \item $A(x)\in \SL_2(\C)$ and $E(x)^2 = 0$.
    \end{enumerate}
\end{proposition}

\begin{proof}
	Using Lemma~\ref{lemma: det(I+t E) formula}, $A(I + tE)\in \SL_2(\C)$ for all $t\in \C$ if and only if $\tr E = \det E = 0$ which occurs if and only if $E^2=0$.
\end{proof}

\subsection{Conditions for winding}
\label{subsection: conditions for winding}
Next proposition describe how to obtain families of invertible cocycles with the winding property.  
Define the seminorm $\Xi:\Mat_2(\R)\to [0,+\infty[$,
$$ \Xi(E):= \max\{\abs{e_{11}-e_{22}}, \, 2\,\abs{e_{12}}, \, 2\,\abs{e_{21}} \} , $$
for $ E=\begin{bmatrix}
	e_{11} & e_{12} \\ e_{21} & e_{22}
\end{bmatrix}$  and consider the following sets where
 $\Delta:\Mat_2(\R)\to \R$  is the function in Lemma~\ref{lemma: det(I+t E) formula}.
\begin{align*}
	\Gamma_+ &:=\left\{ E \in\Mat_2(\R) \colon \Xi(E)>0,\, \Delta_E\geq 0  \, \text{ and }\,  e_{21}\leq 0\leq e_{12} \right\} \\
	\Gamma_- &:=\left\{ E \in\Mat_2(\R) \colon \Xi(E)>0,\, \Delta_E\geq 0 \, \text{ and }\,  e_{12}\leq 0\leq e_{21} \right\}.
\end{align*}
\begin{proposition}
	\label{GL2 criterion}
	For any continuous function  $E\colon X\to \Gamma_\pm$,  the family of cocycles  $A_t:=A+t\, A E$ takes values in $\GL_2(\R)$ and satisfies Assumptions 1-3.
\end{proposition}
\begin{proof}
    Follows from the propositions \ref{pos winding charact}, \ref{pos winding charact, case Delta=0} and \ref{proposition: Gamma characterization}. 
\end{proof}

We now give a similar criterion to obtain 
$\SL_2(\R)$ cocycles with the winding property.
Consider the sets
\begin{align*}
	N_+  &:=\left\{ E \in\Mat_2(\R) \colon \Xi(E)>0,\, E^2=0  \, \text{ and }\,  e_{12}\leq 0\leq e_{21} \right\} , \\
	N_- &:=\left\{ E \in\Mat_2(\R) \colon \Xi(E)>0,\, E^2=0  \, \text{ and }\,  e_{21}\leq 0\leq e_{12} \right\}  ,
\end{align*}
whose union can be characterized as the following conic surface.

\begin{proposition}
\label{cone characterization}
For any matrix $E\in N_-\cup N_+$ there exist unique
real numbers $r\neq 0$ and $\theta\in [0,2\pi[$ such that
	$$ E = \begin{bmatrix}
		-r\,\cos\theta\, \sin\theta & r\,\cos^2\theta\\
		-r\,\sin^2\theta & r\,\cos\theta\, \sin\theta
	\end{bmatrix}  .
	$$
\end{proposition}

\begin{proposition}
	For any continuous function  $E\colon X\to N_\pm$,  the family of cocycles  $A_t:=A+t\, A E$ takes values in $\SL_2(\R)$ and satisfies Assumptions 1-3.
\end{proposition}

\begin{proof}
	Notice that by Proposition~\ref{lemma: SL2 criterion}, we have $E\in N_\pm$ f and only if 
	$E\in \Gamma_\pm$ and $I+t E\in \SL_2(\R)$.
	The conclusion follows then from Proposition~\ref{GL2 criterion}.
\end{proof}

\bigskip

\subsection{Conditions for dominated splitting}
We give some sufficient conditions for Assumption 4 and the weaker alternative  hypothesis $L_1(B)>-\infty$.

\begin{proposition}
	\label{dom splitting criterion}
	Let $X$ be a compact metric space and let $T:X\to X$   be a homeomorphism  preserving a probability measure $\mu\in\Prob(X)$.
	Given $B:X\to \Mat_2(\R)$ continuous such that  $\rank \,B(x)=1$ and $B(Tx)\, B(x) \neq 0 $, for  every $x\in X$,
	then $B$ has dominated splitting and  $L_1(B)>-\infty=L_2(B)$.
\end{proposition}

\begin{proof}
	Any rank $1$ matrix $M\in\Mat_2(\R)$
 	can be written as $M=v\, w^t$ where $v,w\in\R^2$ are  column vectors, which means that
	the action of $M$ on $\R^2$ is described by
	$M\,u=v\,\langle w, u\rangle$. 
	Given a continuous function $B:X\to \Mat_2(\R)$ with rank $1$ values, there exist continuous maps  $v:X\to\R^2$ 
	and $w:X\to \R^2$ such that
	$B(x)=v_x\, w_x^t$, for all $x\in X$. 
	We can assume that $\norm{v_x}=1$ for all $x\in X$, so that $\norm{B(x)}=\norm{w_x}$.
	With this notation 
	\begin{align}
	B^n(x)\, u &= B(T^{n-1}x)\, \cdots \, B(Tx)\, B(x)\, u \nonumber \\
	&= \label{Bn(x) expression}
	\langle w_x, u\rangle  \, \left( \prod_{j=1}^{n-1} \langle w_{T^jx}, v_{T^{j-1}x}\rangle \right)\, v_{T^{n-1}x} .
	\end{align}
	In particular, since  $B(T x)\, B(x)= \langle w_x, \cdot\rangle\, \langle w_{T x}, v_x \rangle\, v_{T x}\neq 0$, we get that 
	$\abs{ \langle w_{T x}, v_x \rangle }>0$ for all $x\in X$. By compactness of $X$, this function admits a positive lower bound $c>0$.	
	Therefore, by Birkhoff's Theorem, for $\mu$-almost every point $x\in X$
	\begin{align*}
	L_1(B) &= \lim_{n\to\infty} \frac{1}{n}\, \log \norm{B^n(x)} \\
	&=  \lim_{n\to\infty}  \frac{1}{n-1}\, \sum_{j=1}^{n-1}  \log \abs{\langle w_{T^jx} ,v_{T^{j-1}x} \rangle } \\
	&= \int_X \log \abs{ \langle w_{Tx}, v_x\rangle  }\, d\mu(x) \geq \log c >-\infty  .
	\end{align*}
	Using~\eqref{Bn(x) expression}, we can easily determine the Oseledets decomposition $\R^2=E_0(x)\oplus E_\infty(x)$ of the cocycle $B$. The subspace $E_0(x)$  is the linear span of 
	 $ v_{T^{-1}x}$, associated with the first Lyapunov exponent, while $E_\infty(x)=w_x^\perp$ is associated with $L_2(B)=-\infty$.
	Finally, because $\norm{ B(x)\vert_{E_0(x)}}=\abs{\langle w_x, v_{T^{-1} x} \rangle}\geq c>0$ and $\norm{ B(x)\vert_{E_\infty(x)}  } =0$ the cocycle $B$ has dominated splitting.
\end{proof}

The following is a sufficient condition for $L_1(B)>-\infty$.
\begin{proposition}
	Let $X$ be a compact metric space and let $T:X\to X$   be a  homeomorphism preserving a probability measure $\mu\in\Prob(X)$.
	Given $B:X\to \Mat_2(\R)$ continuous such that 
	\begin{enumerate} 
		\item $\rank \,B(x)=1$ for $\mu$-almost all $x\in X$,
		\item $\int \log \norm{B(x)}\, d\mu(x)<\infty$,
		\item $\int \log \norm{B(T x)\,B(x)}\, d\mu(x)>-\infty$
	\end{enumerate}
	then  $L_1(B)>-\infty=L_2(B)$.
\end{proposition}

\begin{proof}
	Similar to the proof of  Proposition~\ref{dom splitting criterion}.
\end{proof}

\begin{definition}\label{definition: strictly positively winding}
    We say that a smooth family of cocycles $\{A_t:X\to \GL_2(\R)\}_{t\in I}$ is strictly positively winding over an interval $J\subset I$, if there exist $c>0$, $n_0\in \N$ such that for every $n\geq n_0$, $\hat v\in \mathbb{P}^1$, $t\in J$, $x\in X$,
    \begin{align*}
        \frac{ 
            (A^n_t(x)\,v)\, \wedge\, (\frac{d}{dt}A^n_t(x)\,v)
        } {\norm{A^n_t(x)\,v }^2 } \geq c.
    \end{align*}
    Analogously we define strictly negatively winding. Strictly winding means either strictly positively or else strictly negatively winding.
\end{definition}
\begin{remark}
    Schrodinger families are always strictly positively winding with $n_0=2$ over any compact interval of energies.
\end{remark}

\begin{proposition}
\label{dom splitting corollary}
For any continuous function  $E\colon X\to N_\pm$ such that
$$E(T x)\, A(x)\, E(x)\neq 0\quad \forall \, x\in X,$$
 the family of cocycles  $A_t:=A+t\, A E$ takes values in $\SL_2(\R)$ and satisfies Assumptions 1-4.
 Moreover, for every compact interval $I\subset \R$ the family $\{A_t\}_{t\in I}$ is strictly winding with $n_0 =2$.
\end{proposition}

\begin{proof}
The first statement follows from Proposition~\ref{dom splitting criterion}.
	
To  get the lower bound on the winding speed we use 	Proposition~\ref{winding velocity formula}.
For the sake of concreteness we assume that $E$ takes values in $N_+$ which leads to a positive winding  family $A_t$.
Let $\hat e(x)\in\Pp^1$ be the direction of $\Ker(E(x))$,
which is also the range of $E(x)$.
By hypothesis $\hat e(T x)\neq \hat A(x)\, \hat e(x)$, for all $x\in X$. Hence, by continuity, and compactness of $X$, there exists $r>0$ such that
$\hat A(x)\, B(\hat e(x), r) \cap B(\hat e(T x), r)=\emptyset$,
for all $x\in X$, where $B(\hat v, r)$ denotes the ball of radius $r$ centered at $\hat v$ in $\Pp^1$. By the positive winding property there exists $\beta>0$ such that
$$(A_t(x)\, v)\, \wedge(\dot A_t(x)\, v)= v\, \wedge  E(x)\, v\geq \beta$$
 for all $(t,x)\in\R\times X$ and every unit vector $v$ with $\hat v\notin B(\hat e(x),r)$. Let 
 $$C=\max_{(t,x)\in I \times X}\norm{A_t(x)}.$$
 Fixing $x\in X$ and a unit vector $v\in \R^2$, we write
$A_{j,t}:=A_t(T^{j-1} x)$ and
$\vb_{j}=\vb_{j}(t):= A^j_t(x)\, v/ \norm{A^j_t(x)}$,
for all $j\in\N$.
Then since every summand in the conclusion of   formula provided in  Proposition~\ref{winding velocity formula} is non-negative, the last two terms are enough to get the desired positive lower bound. In fact since either $\vb_{n-1}(t)\notin B( \hat e(T^{n-1} x), r)$ or else
$\vb_{n}(t)\notin B( \hat e(T^{n} x), r)$ we have	
\begin{align*}
&\frac{   (A^n_t(x)\,v)\, \wedge\, (\frac{d}{dt}A^n_t(x)\,v) } {\norm{A^n_t(x)\,v }^2 } \\
&\qquad = \frac{ (A_{n,t} \vb_{n-1} )\wedge (\dot A_{n,t} \vb_{n-1})}{\norm{A_{n,t} \vb_{n-1}}^2} \\
&\qquad\qquad + \frac{1}{\norm{A_{n,t}\, \vb_{n-1}}^2}\, \frac{ (A_{n-1,t} \vb_{n-2} )\wedge (\dot A_{n-1,t} \vb_{n-2})}{\norm{A_{n-1,t} \vb_{n-2}}^2} + \cdots \\
&\qquad  \geq \max\left\{ C^{-2}\,\beta,  C^{-4}\,\beta\right\}
=  C^{-4}\,\beta 
\end{align*} 
which concludes the proof with $c:=C^{-4}\beta$.
\end{proof}

From now on until the end of this section we focus on the random case. Let
$X:=\{1,\ldots, \kappa\}^\Z$ be the space of sequences in $\kappa$ symbols, and let  $T:X\to X$ be the Bernoulli shift in $X$ equipped  with some Bernoulli probability  measure $\mu=(p_1,\ldots, \kappa)^\Z$, where $p_1+\cdots + p_\kappa=1$ and $p_j>0$ for $j=1,\ldots, \kappa$. 

A random or locally constant cocycle $\bld{A}:X\to \SL_2(\R)$ is determined by
a vector of $\kappa$ matrices $\underline A=(A_1,\ldots, A_\kappa)\in \SL_2(\R)^\kappa$, via the formula $\bld{A}(\omega):=A_{\omega_0}$,
where $\omega=(\omega_j)_{j\in\Z}$. This information is usually gathered in the form of a finitely supported measure
$$  \mu(\underline A):=\sum_{j=1}^\kappa p_j\, \delta_{A_j} \in\Prob(\SL_2(\R)) .$$

Assume also that
$\underline A\in\SL_2^\ast(\R)^\kappa$, where $\SL_2^\ast(\R):=\SL_2(\R)\setminus\{\pm I\}$.

\begin{definition}
\label{def Npm(underline A)}
 $N_\pm(\underline A)$ is the space of 
$\underline E=(E_1,\ldots, E_\kappa)\in \Mat_2(\R)^\kappa$ such that respectively $E_j\in N_+$  for all $j$ and
$E_j\in N_-$  for all $j$,
and $E_i\, A_j\, E_j\neq 0$ for all $i,j$.

\end{definition}

\begin{proposition}
Given $\underline A\in\SL_2^\ast(\R)^\kappa$, the set $N_\pm(\underline A)$
is open, dense and full measure in $(N_\pm)^\kappa$.
\end{proposition}

\begin{proof}
By Proposition~\ref{cone characterization}, $N_-\cup N_+$ is a conic surface: its elements are determined by the sign $\pm$, the direction of the kernel and their norm.
Let $\hat v_1,\ldots, \hat v_\kappa$ represent the kernels of
the matrices $E_j$ in $\underline E$.
Then the condition  $E_i\, A_j\, E_j\neq 0$ translates to
$\hat v_i\neq \hat A_j \hat v_j$. Since $A_j\neq \pm I$,
each equation $\hat v_i = \hat A_j \hat v_j$ determines a codimension $1$ submanifold of the $2\kappa$-dimensional manifold $(N_\pm)^\kappa$. Therefore $N_\pm(\underline A)$ is the complement of a finite union of regular hypersurfaces,
which shows that it is open, dense and full measure in $(N_\pm)^\kappa$.
\end{proof}

Given $\underline A=(A_1,\ldots, A_\kappa)\in\SL_2(\R)^\kappa$
and  $\underline E=(E_1,\ldots, E_\kappa)\in N_\pm(\underline A)$ we may consider the parameterized line of random cocycles
$$ \underline A_t= \underline A\,(I+t \underline E):=(A_1\,(I+t E_1), \ldots, A_\kappa\, (I+t E_\kappa)). $$

\begin{proposition}
	\label{special case random cocycles satisfying 1-4}
	If $\underline A\in\SL_2^\ast(\R)^\kappa$ and
	 $\underline E\in N_\pm(\underline A)$, then the family of random cocycles  $\underline A_t:=\underline A\, (I+t \underline E)$ 
takes values in $\SL_2(\R)$ and satisfies Assumptions 1-4.
Moreover $\underline A_t$ is strictly winding with $n_0=2$.
\end{proposition}

\begin{proof} Since 
$E_i\, A_j\, E_j \neq 0$ for all $i,j$, the conclusion follows by Proposition~\ref{dom splitting corollary}.
\end{proof}

%%%%%%%%%%%%%%%%%%%%%%%%%%%%%%%%%%%%%%%%%%%%%%%%%%%%%%%%%%%%%

\section{Applications}\label{consequences}
In this section we provide a list of applications of the main theorem.

\subsection{Regularity of the rotation number}

Just like the IDS, the  fibered rotation number has
a minimal modulus of continuity in the same spirit of Craig and Simon~\cite{CraigSimon}.

\begin{proposition}
	The function $\rho$ is log-H\"older continuous. More precisely, for any $t\in \R$ and any $s\in \R$ with $\abs{t-s}<\epsilon<1$, there exist a constant $C=C(t,\epsilon)$ which tends to zero as $\epsilon \to 0$, such that
	$$
	\lvert \rho(t)-\rho(s)\vert \leq \frac{C}{\log\frac{1}{\lvert t-s\rvert} }.
	$$
\end{proposition}

\begin{proof}	
	Since the support of $d\rho$ is bounded, we have 
	$$
	0\leq \int_{\lvert t-s\rvert\geq 1}\log\lvert t-s\rvert d\rho(s) <+\infty, 
	$$
	Since the cocycle $A_t$ is uniformly bounded for $t$ in the support of $d\rho$ and $L_1(B)> -\infty$,  it follows by~\eqref{Thouless ID} that
	$$
	\int_{\lvert t-s\rvert<1}\log\frac{1}{\lvert t-s\rvert}d\rho(s) =L_1(B)-L_(A_t) + \int_{\lvert t-s\rvert\geq 1}\log\lvert t-s\rvert d\rho(s)  <+\infty .
	$$
	Thus we have
	$$
	C(t,\epsilon):=\max\left\{ \int_{t}^{t+\epsilon}\log\frac{1}{\lvert t-s\rvert}d\rho(s), \,
	\int_{t-\epsilon}^{t}\log\frac{1}{\lvert t-s\rvert}d\rho(s) \right\}<+\infty.
	$$
	Finally,  observe that $C(t,\epsilon)\geq \lvert \rho(t)-\rho(s)\rvert\, \log \frac{1}{\lvert t-s\rvert} $ when $|t-s|<\epsilon$.
\end{proof}

\begin{proposition}
	Let $A_t:X\to\SL_2(\R)$, $A_t(x)=A(x)\, (I+t\, E(x))$, be a family of cocycles under the assumptions of Theorem~\ref{main}.
	For any open interval $I\subset \R$, the following are equivalent:
	\begin{enumerate}
		\item $t\mapsto L_1(A_t)$ is analytic on $I$;
		\item $t\mapsto \rho(A_t)$ is constant on $I$.
	\end{enumerate}
\end{proposition}

\begin{proof}
Follows from Thouless formula.	
\end{proof}

%A list of settings where a known regularity of the Lyapunov exponent implies a similar regularity for the fibered rotation number 

By Goldstein and Schlag~\cite{GS01},  under the assumptions of Theorem \ref{main} a \textit{good enough} regularity of the Lyapunov exponent $L_1(A_t)$ transfers over to the same regularity for the rotation number $\rho(A_t)$. As mentioned in Section~\ref{intro}, this good enough regularity of the LE is indeed available for a wide class of linear cocycles, which then implies  similar continuity  properties for the corresponding fibered rotation number and establishes Proposition~\ref{conseq prop}.

\medskip

\subsection{A Johnson-type theorem}

A classical result of R. Johnson ~\cite{Jo86}
states that given a family of Schr\"odinger cocycles $A_E$, over some ergodic transformation,   $A_E$ is uniformly hyperbolic if and only if $E$ lies inside a gap of the spectrum of the corresponding Schr\"odinger operator, i.e., the IDS is locally constant around $E$.
In~\cite{ABD12} Avila, Bochi and Damanik considered a continuous cocycle 
$A\in C^0(X,\SL_2(\R))$ homotopic to a constant and the winding family of cocycles   $R_\theta\, A$, where $R_\theta$ denotes the rotation by angle $\theta$. In the spirit of Johnson's theorem
they have proved that $R_\theta\, A$ is uniformly hyperbolic if and only if  the fibered rotation number $\rho(R_\theta\, A)$ is locally constant around $\theta$. See~\cite[Proposition C.1]{ABD12}.
More recently, in~\cite[Theorem A.9]{GK21} Gorodetski and Kleptsyn have generalized this result to a context that basically matches our own. The following is a corollary of their theorem A.9.

\begin{theorem}[Gorodetski, Kleptsyn]
Let $A_t:X\to\SL_2(\R)$, $A_t(x)=A(x)\, (I+t\, E(x))$, be a family of cocycles under the assumptions of Theorem~\ref{main} and such that
\begin{equation}
\label{twist cond}
A(x)\, \Ker(E(x))\neq \Ker(E(T x))\quad \forall\, x\in X .
\end{equation}
For any open interval $I\subset \R$, the following are equivalent:
\begin{enumerate}
	\item the cocycle $A_t$ is uniformly hyperbolic for $t\in I$;
	\item $t\mapsto \rho(A_t)$ is constant on $I$.
\end{enumerate}
\end{theorem}

\begin{proof}
The implication (1)$\Rightarrow$(2) follows by
	Lemma~\ref{Johnson trivial half}.

For the converse implication, (2)$\Rightarrow$(1), consider the family of cocycles $A_t^2(x)=A_t(T x)\, A_t(x)$. For the sake of simplicity we assume that $A_t$ is positively winding. Assumption~\eqref{twist cond} implies that
$E(T x) A(x)  E(x)\neq 0$.  Hence, by Proposition~\ref{dom splitting corollary}, there exists $c>0$ such that  
$$ \frac{d}{dt}\frac{A_t(x) v}{\norm{A_t(x) v}} =  \frac{   (A^n_2(x)\,v)\, \wedge\, (\frac{d}{dt}A^2_t(x)\,v) }   {\norm{A^n_t(x)\,v }^2 } \geq c \quad \forall\, (x,t,v)\in X\times\R\times \Su^1 .$$
This ensures the strict increasing assumption of Theorem A.9 in~\cite{GK21}, and the conclusion (1) follows from this theorem.
\end{proof}

The need for assumption~\eqref{twist cond} is justified by the next example.

\begin{example}
Choose $E\in N_+$ and a hyperbolic matrix $H\in\SL_2(\R)$ such that $H$ a $E$ share a common eigen-direction. Consider the random cocycle generated by the vector $\underline A=(H, I)$ with probabilities $(\frac{1}{2}, \frac{1}{2})$.
The family of cocycles $\bld{A}_t:= \bld{A}\, (I+t\, E)$ satisfies:
\begin{enumerate}
	\item Assumption~\eqref{twist cond} does not hold;
	\item $\bld{A}_t$ is not uniformly hyperbolic. It contains the parabolic matrix $I+t E$;
	\item $\rho(\bld{A}_t)$ is constant. There is a common direction fixed by all  $\bld{A}_t$;
%	\item $L_1(\bld{A}_t)=\frac{1}{2}\, \log \abs{\lambda(H)}$, where $\lambda(H)>1$ is an eigenvalue of $H$.
\end{enumerate} 
\end{example}

\medskip

\subsection{Typical regularity of the Lyapunov exponent}

This last section focuses on the regularity of the Lyapunov exponent of random, locally constant,  $\SL_2$ valued cocycles. Each of these cocycles $\bld A:X\to \SL_2(\R)$ is determined by a matrix vector $\underline A\in\SL_2^\ast(\R)^\kappa$ as well as a  probability vector $(p_1,\ldots, p_\kappa)$. 
Recall that $\SL_2^\ast(\R):=\SL_2(\R)\setminus\{-I,I\}$, as well as Definition~\ref{def Npm(underline A)} where the set
$N_\pm(\underline A)$ of `good' directions was introduced.

For each $A = (A_1,\ldots, A_{\kappa})\in \SL_2(\R)^{\kappa}$, we denote by $\Gamma(\underline A)$ the semigroup generated by the matrices $\{A_1,\ldots, A_{\kappa}\}$. Given a hyperbolic matrix $B\in \SL_2(\R)$ we write $\hat u(B)\in \mathbb{P}^1$ and $\hat s(B)\in \mathbb{P}^1$ to denote respectively the  unstable and stable eigen-directions of $B$.

We say that $\underline A\in \SL_2(\R)^{\kappa}$ has a \emph{heteroclinic tangency} if there exist matrices $B,\, C,\, A\in \Gamma(\underline A)$ such that $A$ and $B$ are hyperbolic and $C\, \hat u(B) = \hat s(A)$. In this case we also say that $(B,\, C,\, A)$ is a tangency for $\underline A$. Recall Definition \ref{definition: strictly positively winding}.

\begin{theorem}\label{breakRegularity}
    Let $\{\bld A_t:X\to \SL_2(\R)\}_{t\in I}$ be a smooth family of cocycles such that $\bld A_t$ is strictly positively winding over $I$, for some $t_0\in I$, $\underline A_{t_0}$ is irreducible and has a heteroclinic tangency. If $\alpha> \frac{H(\mu)}{L_1(\bld A_{t_0})}$, then the fibered rotation number $I\ni t \mapsto \rho(\bld A_t)$
    is not $\alpha$-H\"older continuous around $t = t_0$.
\end{theorem}

The proof of Theorem \ref{breakRegularity} is presented below.

\begin{corollary}\label{corollary:low regularity rotation number}
    Let $\{\bld A_t:X\to \SL_2(\R)\}_{t\in I}$ be a smooth family of cocycles such that $\bld A_t$ is strictly positively winding over $I$, for some $t_0\in I$, $\underline A_{t_0}$ is not uniformly hyperbolic. If $\alpha> \frac{H(\mu)}{L_1(\bld A_{t_0})}$, then the fibered rotation number $I\ni t \mapsto \rho(\bld A_t)$
    is not $\alpha$-H\"older continuous around $t = t_0$.
\end{corollary}
\begin{proof}
    By assumption, $\underline A_0=\underline A$ is not uniformly hyperbolic. Hence by~\cite[Propositions 7.8 and 7.12]{BeDu22} there exist $t_0$ arbitrary close to $0$ such that $\underline A_{t_0}$ admits a heteroclinic tangency and is irreducible. By Theorem~\ref{breakRegularity} the fibered rotation number $\rho(\underline A_t)$ is not $\alpha$-H\"older 
    continuous around $t_0$.
\end{proof}

Next result is a reformulation of
Theorem~\ref{intro low regularity thm}.
\begin{corollary}
\label{regularity dichotomy}
Let $\underline A=(A_1,\ldots, A_\kappa)\in \SL_2^\ast(\R)^\kappa$ be a random cocycle  such that $L_1 (\underline A) > 0$ but $\underline A$ is not uniformly hyperbolic. For all  $ \underline E  \in N_\pm(\underline A)$ and $\alpha >  \frac{H(\mu)}{L_1 (\underline A)}$, the family of random cocycles $\underline A_t := \underline A\, (I + t \underline E_t)$ has Lyapunov exponent $t \mapsto L_1 (\bld A_t)$ which is not $\alpha$-H\"older continuous near $t=0$.
\end{corollary}
\begin{proof}
    As explained in Subsection~\ref{statements}, by~\cite[Lemma 10.3]{GS01} and Theorem~\ref{main},
    the Lyapunov exponent $L_1(\underline A_t)$ has the same regularity as the fibered rotation number. The conclusion follows by Corollary \ref{corollary:low regularity rotation number}.
\end{proof} 

\begin{definition}%Jamerson
\label{def:generalWinding}
    Given $\gamma>0$, we say that a pair of matrices $(B,\, A)\in \SL_2(\R)^2$ has a $\gamma$-matching if there exists a pair of points  $\hat e_1\, \hat e_2 \in \mathbb{P}^1$ such that
    \begin{enumerate}
        \item $A\, B\, \hat e_1 = \hat e_2$;
        
        \item $e^{\gamma}
        \leq \norm{B\, e_1}
        \leq \norm{B} 
        \leq 2\, e^{\gamma}$;
        
        \item $e^{\gamma}
        \leq \norm{A^{-1}\, e_2}
        \leq \norm{A} 
        \leq 2\, e^{\gamma}$.
    \end{enumerate}
    In this case we say that $(B,\, A)$ connects $\hat e_1,\, \hat e_2 \in \mathbb{P}^1$.
\end{definition}
\begin{remark}
Notice that given $\hat e_1\in \Pp^1$, we can define $\hat e_2$ so that (1)-(2) hold, but then (3) will only be  satisfied  for  very particular choices of $(B,\, A)$.
Given $1$-parameter smooth families $A_t, B_t\in\SL_2(\R)$, the correct way to find $\gamma$-matchings $(B_t,\, A_t)$ is to fix first  $\hat e_1, \hat e_2\in\Pp^1$ and then solve the equation 
$B_t\, \hat e_1=A_t^{-1}\hat e_2$ in $t$, looking   for parameters $t$ such that both norms $\norm{B_t\, e_1}$ and $\norm{A_t^{-1} e_2}$ are large of order $e^\gamma$.
\end{remark}

For families of cocycles we have the following definition of matching.
\begin{definition}%Jamerson
    Given a family of cocycles $\{\bld A_t\}_t$, $\gamma>0$, $k\geq 2$ and $t_0\in I$ we say that a sequence $\omega\in X$ is a $(\gamma,\, k,\, t_0)$-matching if there exists $1 \leq m < k$ such that $(\bld A^m_{t_0}(\omega),\,\bld A^{k-m}_{t_0}(T^m\, \omega))$ has a $\gamma$-matching.
\end{definition}

\begin{remark} 
	When $(\bld A^m_{t_0}(\omega),\,\bld A^{k-m}_{t_0}(T^m\, \omega))$ connects the vectors $\hat{\textbf{e}}_1,\, \hat{\textbf{e}}_2$ of the canonical basis of $\R^2$, the above definition of matching agrees with the definition of matching in Section 6 of \cite{BeDu22} with $\gamma := \log \delta^{-1}$.	
\end{remark}

\begin{proposition}%Jamerson
\label{prop:matchingsToTurns}
     There exists $\gamma_0$ and constant $c_*>0$ such that for every positively smooth  strictly  winding family of cocycles $\{\bld A_t\}_{t\in I}$ if $\omega\in X$ is a $(\gamma,\, k,\, t_0)$-matching, then $\bld A^k_t(\omega)$ winds once around $\mathbb{P}^1$ as $t$ ranges in $[t_0 - 4\, c_*^{-1}\, e^{-\gamma},\, t_0 + 4\, c_*^{-1}\, e^{-\gamma}]$.
\end{proposition}
\begin{proof}
    Consider $1\leq m < k$  and directions $\hat e_1,\, \hat e_2 \in \mathbb{P}^1$ which are connected by $(\bld A^m_{t_0}(\omega),\, \bld A^{k-m}_{t_0}(T^m\, \omega))$. Set
    \begin{align*}
        J_0 := [t_0 - 3c_*^{-1}\, e^{-\gamma},\, t_0 + 3c_*^{-1}\, e^{-\gamma}]
    \end{align*}
    
    For any pair of vectors $\hat v,\, \hat w\in \mathbb{P}^1$ satisfying
    \begin{align*}
        d(\hat v,\, \hat \ledir(\bld A^m_{t_0}(\omega))  ) \geq e^{-\gamma}
        \quad
        \text{and}
        \quad
        d(\hat w,\,  \hat \medir^*(\bld A_{t_0}^{k-m}(T^m\, \omega)) ) \geq e^{-\gamma},
    \end{align*}
    we can use Proposition \ref{prop:matchingsToAlmostTurn} to conclude that $d(\bld A_{t_0}^m(\omega)\, \hat v,\, \bld A_{t_0}^{-m}(T^k\, \omega)\, \hat w) \leq 3\, e^{-\gamma}$. If we define $f_+,\, f_-:J_0\to \mathbb{P}^1$ by,
    \begin{align*}
        f_+(t) := \bld A_t^m(\omega)\, \hat v
        \quad
        \text{and}
        \quad
        f_-(t) := \bld A_t^{-m}(T^k\, \omega)\, \hat w,
    \end{align*}
    then by the  strict winding hypothesis, Definition \ref{definition: strictly positively winding},  which we can assume to be positive, we have that there exists $c_*>0$ such that $f'(t) < -c_* < 0 < c_* < f'(t)$, for every $t\in  I$. Since $d(f_+(t_0),\, f_-(t_0)) \leq 3\, e^{-\gamma}$, there exists $t_*\in J_0$ such that
    \begin{align*}
        \bld A_{t_*}^m(\omega)\, \hat v
        = f_+(t_*) = f_-(t_*)
        = \bld A_{t_*}^{-m}(T^k\, \omega)\, \hat w.
    \end{align*}
    Using Corollary \ref{corollary local winding}, this implies that for every $\hat v \notin \text{B}(\ledir(\bld A_{t_0}^m(\omega)),\, e^{-\gamma})$,
    \begin{align*}
        \ell_J( \bld A_t^{k}(\omega)\, \hat v) \geq \pi - e^{-\gamma}.
    \end{align*}
    Writing $J_1 := J_-\cup J \cup J_+$, where $J_- := [t_0 - \frac{7}{2}\, c_*^{-1}\, e^{-\gamma},\, t_0 - 3c_*^{-1}\, e^{-\gamma}]$ and $J^+ := [ t_0 + 3c_*^{-1}\, e^{-\gamma},\, t_0 + \frac{7}{2}\, c_*^{-1}\, e^{-\gamma}]$ and using again the strict positive winding, we have
    \begin{align*}
        \ell_{J_1}(\bld A_t^k\, \hat v)
        &\geq \ell_{J_+}(\bld A_t^k\, \hat v) + \ell_{J_-}(\bld A_t^k\, \hat v) + \ell_{J}(\bld A_t^k\, \hat v)\\
        &\geq  e^{-\gamma} + \pi - e^{-\gamma} = \pi.
    \end{align*}
    In other words, for every $\hat v\notin \text{B}(\ledir(\bld A_{t_0}^m(\omega)), e^{-\gamma})$, the curve $t\in J_1\mapsto \bld A_t^k(\omega)\, \hat v$ gives one turn around $\mathbb{P}^1$. Thus, by Corollary \ref{lemma winds n times} for every $\hat w\in \mathbb{P}^1$ and $\hat v\notin \text{B}(\ledir(\bld A_{t_0}^m(\omega)), e^{-\gamma})$, the equation $\bld A_t^{-m}(T^k\, \omega)\, \hat w = \hat v$ has a solution for some $t\in J_1$ which implies again by Corollary \ref{corollary local winding} that
    \begin{align*}
        \ell_{J_1}(\bld A_t^{-m}(T^k\, \omega)\, \hat w) \geq \pi - e^{-\gamma}.
    \end{align*}
    Taking $J_* := [t_0 - 4c_*^{-1}\, e^{-\gamma},\, t_0 + 4c_*^{-1}\, e^{-\gamma}]$ and using the same argument as above we conclude that for every $\hat w\in \mathbb{P}^1$,
    \begin{align*}
        \ell_{J_*}(\bld A_t^{-m}(T^k\, \omega)\, \hat w) \geq \pi.
    \end{align*}
    Therefore, for every $\hat v,\, \hat w\in \mathbb{P}^1$, the equation
    \begin{align*}
        \bld A_t^{-m}(T^k\, \omega)\, \hat w = \hat v
        \quad \Leftrightarrow \quad
        \quad
        \bld A_t^k(\omega)\, \hat v = \hat w,
    \end{align*}
    has a solution for some $t\in J_*$  one and more  application of Corollary \ref{lemma winds n times}  concludes the proof.
\end{proof}

Let $\{\bld A_t\}_{t\in I}$ be a strictly  positively winding smooth family of cocycles and $J$ be a sub-interval of $I$. We denote by $\Sigma(\gamma,\, k,\, J)$ the set of sequences $\omega\in X$ which are $(\gamma,\, k,\, t)$-matching for some $t\in J$. For each $\delta>0$ we write $J_{\delta} :=J+ [ - \delta,\, \delta]$.
\begin{proposition}%Jamerson
\label{length and sum of chi of Sigma}
    Given $\omega\in X$, $n\geq 1$, for every $\hat v \in \mathbb{P}^1$ if $\delta := 4c_*^{-1}\, e^{-\gamma}$,
    \begin{align*}
        \ell_{J_{\delta}}(\bld A_t^{nk}(\omega)\, \hat v)
        \geq \pi\, \sum_{j=0}^{n-1}\chi_{\Sigma(\gamma,\, k,\, J)}(T^{jk}\, \omega).
    \end{align*}
\end{proposition}
\begin{proof}
    If $T^{jn}\, \omega \in \Sigma(\gamma,\, k,\, t_0)$ for some $t_0\in J$ and $0\leq j \leq n-1$, then by Proposition \ref{prop:matchingsToTurns}, $\bld A^k_t(T^{jn}\, \omega)$ winds once around $\mathbb{P}^1$ as $t$ ranges in the interval $[t_0 - 4c_*^{-1}\, e^{-\gamma},\, t_0 + 4c_*^{-1}\, e^{-\gamma}] \subset J_\delta$. In particular, for every $\hat v\in \mathbb{P}^1$
    \begin{align*}
        \ell_{J_{\delta}}(\bld A_t^k(T^{jn}\, \omega)\, \hat v) \geq \pi.
    \end{align*}
	Combining  Proposition~\ref{adding turns around P} with Corollary~\ref{lemma winds n times} and Lemma~\ref{lemma 1 ell_I}, the stated inequality follows. 
\end{proof}

\begin{corollary}%Pedro
\label{lower bound drho}
    For any interval $J\subseteq I$ and $\gamma>0$, if $\delta:=4c_*^{-1}\, e^{-\gamma}$,
    \begin{align*}
        d\rho(J_{\delta})\,  \geq \frac{1}{k}\, \mu\left(
            \Sigma(\gamma,k,J)
        \right), \quad \forall \, k\in\N .
    \end{align*}
\end{corollary}
\begin{proof}
Apply  Birkhoff's ergodic theorem together with
Lemma~\ref{lem:lengthToRotation} and  Proposition~\ref{length and sum of chi of Sigma}.
\end{proof}

Let $\lambda:= \inf_{|t-t_0|\leq \delta} L(\underline A_t)$ and assume $\underline A_{t_0}$ has a heteroclinic tangency.
\begin{proposition}
\label{lem:120522.2}%Pedro
     For every $\hat v, \hat w\in \mathbb{P}^1$
     and $0<\beta\ll \lambda$, there exist constants $C^\ast, C, c>0$, an infinite subset $\mathbb{N}'\subset \N$, a sequence $(t_l)_{l\in \N'}$ with $|t_l - t_0| \leq C^\ast\, e^{-c\,l^{1/3}}$, and a sequence of measurable sets $\mathcal{M}_l\subset X$ such that defining
\begin{align*}
    I_l := [t_l - C\, e^{-l\, (\lambda - \beta)},\, t_l + C\, e^{-l\, (\lambda -\beta)}],
\end{align*}
 for every $\omega\in \mathcal{M}_l$  we can find $t_l^\ast \in I_l$ such that 
    \begin{enumerate}
        \item $\bld{A}_{t_l^\ast}^{2 l^3+l}(\omega)\, \hat v = \hat w$;
        
        \item $e^{(\lambda-\beta)\, l^3} \leq
         \norm{\bld{A}_{t_l^\ast}^{l^3}(\omega)\, v}\leq \norm{\bld{A}_{t_l^\ast}^{l^3}(\omega)}\leq e^{ (\lambda+\beta)\, l^3}$;
         
        \item $e^{(\lambda-\beta)\, l^3} \leq \norm{\bld{A}_{t_l^\ast}^{-l^3}(T^{2l^3+l}\, \omega)\, w}\leq \norm{\bld{A}_{t_l^\ast}^{-l^3}(T^{2l^3+l}\, \omega)}\leq e^{(\lambda+\beta)\, l^3}$;
        
        \item $\norm{\bld{A}_{t_l^\ast}^{2 l^3+l}(\omega)\, v}\leq e^{3 \,\beta\, l^3}$.
    \end{enumerate}  
    Moreover,
    \; $\displaystyle \mu\left(
            \mathcal{M}_l
        \right) \geq (1/2 - \beta)^2\, e^{
            -l\,(H(\mu)+ \beta)
        }$.
\end{proposition}
\begin{proof}
This proposition is a reformulation of~\cite[Lemma 8.1]{BeDu22}, which
holds for positively winding $\SL_2(\R)$-cocycles. As indicated in~\cite[Figure 1]{BeDu22}, the proof of Lemma 8.1  uses the strict positive winding 
property through Proposition 7.13, this being the only place where the fact that the cocycle  comes from a Schr\"odinger cocycle is used.
All other statements hold for general strictly  positively winding smooth families of  cocycles. 
\end{proof}

\begin{remark}
\label{matching remark}
Proposition~\ref{lem:120522.2} says that for every $l\in \N'$ and every $\omega\in \mathcal{M}_l$, 
$\left( \bld{A}_{t_l^\ast}^{l^3}(\omega)\,  ,\, \bld{A}_{t_l^\ast}^{l^3+l}(T^{l^3+l}\omega)\, \right)$ has a $[(\lambda-\beta)\, l^3]$-matching. In particular,
\begin{align*}
    \mathcal{M}_l \subset \Sigma((\lambda-\beta)\, l^3,\, 2l^3 + l,\, I_l).
\end{align*}
\end{remark}

\begin{proof}
[Proof of Theorem~\ref{breakRegularity}]%Pedro
\label{proofThm:breakRegularity}
Given $\alpha>H(\mu)/L_1(\underline A_{t_0})$, choose $\delta>0$ small enough so that
$\lambda:= \inf_{|t-t_0|\leq \delta} L(\underline A_t)$ satisfies $\alpha\, \lambda-H(\mu)>0$
and then take $\beta>0$ small so that
 $\alpha\, \lambda-H(\mu)>2\,\beta+\alpha\, \beta$.
This implies that 
\begin{align}
\label{beta inequality}
  -H(\mu)-\beta+  \alpha \, (\lambda-\beta) > \beta .
\end{align}
From Proposition~\ref{lem:120522.2} take $l\in \N'$, consider the set $\mathcal{M}_l\subset X$ and the interval $I_l\subset \R$ therein of length
$|I_l|=2\, C\, e^{- l\,(\lambda-\beta)}$.
Defining $\gamma_l:=l^3\, (\lambda-\beta)$,
by Remark~\ref{matching remark} we have
$\mathcal{M}_l\subseteq \Sigma(\gamma_l, 2 l^3+l, I_l)$. Let 
$\delta_l:=6\, c_\ast^{-1}\, e^{-\gamma_l}$
be the constant associated with $\gamma_l>0$ in 
Corollary~\ref{lower bound drho} and set
$\tilde I_l:= I_l+[-\delta_l,\delta_l]$, so that
$|\tilde I_l|\sim |I_l|\sim e^{-\gamma_l}$.

By Corollary~\ref{lower bound drho} and Proposition~\ref{lem:120522.2}
\begin{align*}
    d\rho\left(
        \tilde I_l
    \right)
    &\geq \frac{1}{2l^3 + l}\, \mu\left(
        \Sigma(\gamma_l,\, 2l^3 + l,\, I_l)
    \right)
    \geq \frac{1}{2l^3 + l}\, \mu(\mathcal{M}_l)\\
    &\geq \frac{1}{4\,(2l^3 + l)}(1/2-\beta)^2\, e^{-l\,(H(\mu) + \beta)}.
\end{align*}
Thus, using~\eqref{beta inequality} we get
\begin{align*}
 \frac{
        d\rho\left( \tilde I_l \right)
    }{|\tilde I_l| ^{\alpha}}
   \gtrsim  e^{
        l\, \left(
            -H(\mu) - \beta + \alpha(\lambda - \beta)
        \right)} \gtrsim  e^{\beta\, l} .
\end{align*}
Taking $l\to \infty$ we conclude that $\rho$ can not be $\alpha$-H\"older continuous. 
\end{proof}

Given $\beta>0$, the function $f:I\subset \R\to\R$ is called $\beta$-log-H\"older continuous if
there exists $C<\infty$ such that for all $t,t'\in I$,
$$ \abs{f(t)-f(t')} \leq C\,\frac{1}{\log^\beta\left( \abs{t-t'}^{-1} \right)}  . $$

\begin{theorem}
Consider the random linear cocycle generated
by the matrices $\underline A=(C,D)$ 
\begin{equation}
	C:= \begin{bmatrix} 
		0 & -1 \\ 1 & 0
	\end{bmatrix}\quad \text{ and } \quad
	D:= \begin{bmatrix} 
		e & 0 \\ 0 & e^{-1}
	\end{bmatrix}
\end{equation}
with probability vector $(\frac{1}{2},\frac{1}{2})$.
For any $\underline E\in N_\pm(\underline A)$,
the Lyapunov exponent function
$t\mapsto L_1(\underline A\,(I+t \underline E))$ 
is not $\beta$-log-H\"older continuous around $t=0$ for any $\beta>3$.
\end{theorem}

\begin{proof}
In~\cite{DKS19} it was shown that there exists a $1$-parameter curve  of cocycles $A_t$ passing through $\underline A=(C,D)$ at $t=0$, along which $L_1(\bld A_t)$ has a `nasty' modulus of continuity. The meaning of `nasty' is clarified below. The strategy was to conjugate $\underline A$ to a Schr\"odinger cocycle over a mixing Markov shift, so that, up to conjugation, $\bld A_t$ could be viewed as a 
family of Schr\"odinger cocycles, where $t$ is the system's energy. It was proved in~\cite[Theorem 1]{DKS19} that the IDS of the
corresponding Schr\"odinger operator is not $\beta$-log-H\"older continuous for any $\beta>2$. Recall that, the IDS is the fibered rotation number of this positively winding Schr\"odinger family. The  loss of regularity comes from the existence of many matching configurations. See Lemma 1 and Proposition 8 in~\cite{DKS19}.

Given $\underline E\in N_\pm(\underline A)$, by Proposition~\ref{special case random cocycles satisfying 1-4} the cocycle
$\underline A_t:= \underline A\,(I+t \underline E)$ satisfies assumptions 1-4 of Theorem~\ref{main} and whence also the Thouless formula~\eqref{Thouless ID}. An adaption of the argument in~\cite{DKS19}
 based on the  Lemma 1 mentioned above, gives the following lemma.
Actually some technical aspects of the proof get simplified because there is no need anymore to conjugate the original cocycle to a Schr\"odinger one over a mixing Markov shift.

\begin{lemma}
The function $\R\ni t\mapsto \rho(\underline A_t)$ is not 
$\beta$-log-H\"older continuous for any $\beta>2$.
\end{lemma}

From the Thouless formula~\eqref{Thouless ID} it follows that  $L_1(\underline A_t)$ is essentially the Hilbert transform of the fibered rotation number $\rho(\underline A_t)$. The Hilbert transform and its inverse are examples of singular integral operators. M. Goldstein and W. Schlag~\cite[Lemma 10.3]{GS01} proved that any singular integral
operator on a space of functions preserves certain modulus of continuity, which include the H\"older and weak-H\"older but not the $\beta$-log-H\"older modulus of continuity. In a recent paper, Avila, Last, Shamis and Zhou have improved this result showing that for any $\beta>2$, if $L_1(A_t)$ is $\beta$-log-H\"older then  $\rho(A_t)$ is $(\beta-1)$-log-H\"older , see Proposition 2.2 and Corollary 2.3 in~\cite{ALSZ21}. Whence, in view of the previous lemma, 
 $L_1(\underline A_t)$ can not be
$\beta$-log-H\"older continuous for any $\beta>3$. 
\end{proof}

\section{Appendix: linear algebra facts}%Pedro

\subsection{Projective analysis}
\label{appendix:projective analysis} 
Given $A\in\Mat_2(\R)$
denote by $\{\medir(A),\, \ledir(A)\}$  an orthonormal set of singular directions of $A$,  defined by the relations 
\begin{align*}
    (A^t\,A)\, \medir(A)=\norm{A}^2\, \medir(A)
    \quad
    \text{ and }
    \quad
    (A^t\,A) \,\ledir(A)=\conorm(A)\,\ledir(A) ,  
\end{align*}
where  $\conorm(A)$  denotes the co-norm of $A$,  
$\conorm(A):=\norm{A^{-1}}^{-1}$ if $A$ is invertible 
and otherwise $\conorm(A):=0$.
Define also
\begin{align*}
    \medir^\ast(A) := \medir(A^t)
    \quad
    \text{and}
    \quad
    \ledir^\ast(A) := \ledir(A^t),
\end{align*}
so that
\begin{align*}
    A\, \medir(A)=\norm{A}\,\medir^\ast(A)
    \quad
    \text{and}
    \quad
    A\,\ledir(A)= \conorm(A)\,\ledir^\ast(A).
\end{align*}
For the sake of simplicity we use the same notation for the singular vectors and its projectivization. For each pair of vectors $v,\, w\in \R^2$ we write 
\begin{align*}
    d(\hat v,\, \hat w) := \frac{|v\wedge w|}{\norm{v}\, \norm{w}} = \sin \measuredangle(v, w),
\end{align*}
for the usual distance in $\mathbb{P}^1$.
\begin{lemma}\label{L1}
    Take $A\in \SL_2(\R)$. For each $\hat v\in \mathbb{P}^1$,
    \begin{align*}
        d(\hat v,\, \hat \ledir(A)) = \sqrt{
            \frac{
                \norm{A\, v}^2 - \norm{A}^{-2}
            }{\norm{A}^2 - \norm{A}^{-2}}        
        }.
    \end{align*}
    In particular,
    \begin{align*}
        \frac{1}{\norm{A}}\sqrt{
            \norm{A\, v}^2 - \norm{A}^{-2}
        }
        \leq d(\hat v,\, \hat \ledir(A))
        \leq \frac{\norm{A\, v}}{\norm{A}}.
    \end{align*}
\end{lemma}

\begin{proof}
Let $v= a\, \medir(A) + b\, \ledir(A) $ be a unit vector, i.e., $a^2+b^2=1$, with $a>0$. Then
$a=d(\hat v, \hat \ledir(A) $
and $A\, v=a\, \norm{A}\, \medir^\ast(A) + b\, \norm{A}^{-1}\, \ledir^\ast(A)$, which implies
that
$$ \norm{A\, v}^2 = a^2\, \norm{A}^2 + (1-a^2)\, \norm{A}^{-2} .$$
Solving in $a=d(\hat v, \hat \ledir(A) $ yields the conclusion.
\end{proof}

\begin{lemma}\label{L2}
    For any $\hat v\in \Pp^1$,
    \begin{align*}
        d(A\hat v,\, \hat \medir^*(A))  \leq \frac{1}{d(\hat v,\,  \hat \ledir(A)} )\norm{A}^2.
    \end{align*}
\end{lemma}

	\begin{proof}
	Using the setup of Lemma~\ref{L1}'s proof,
	$$ d(\hat v, \hat \medir^\ast(A))=\frac{|A v \wedge \medir^\ast(A) |}{\norm{A v}}
	\leq \frac{ \sqrt{1-a^2} }{\norm{A v}\, \norm{A}} \leq \frac{1}{d(\hat v,\,  \hat \ledir(A) )\norm{A}^2} $$
	where in the last inequality we use Lemma~\ref{L1}.
	\end{proof}

 In  the next proposition we use  Definition~\ref{def:generalWinding}.
\begin{proposition}\label{prop:matchingsToAlmostTurn}
    There exists $\gamma_0>0$ such that for every $\gamma \geq \gamma_0$, if $(B,\, A)\in \SL_2(\R)^2$ is a $\gamma$-matching and $\hat v,\, \hat w\in \mathbb{P}^1$ satisfy
    \begin{align*}
        d(\hat v,\, \hat \ledir(B)) \geq e^{-\gamma}
        \quad
        \text{and}
        \quad
        d(\hat w,\, \hat \medir^*(A)) \geq e^{-\gamma},
    \end{align*}
    then
    \begin{align*}
        d(B\, \hat v,\, A^{-1}\, \hat w) \leq 3\, e^{-\gamma}.
    \end{align*}
\end{proposition}
\begin{proof}
	Assume $(B,A)$ connects $\hat e_1$, $\hat e_2\in\Pp^1$. By Lemma \ref{L2} and the matching assumption,
    \begin{align*}
        d(B\, \hat v,\, \hat \medir^*(B)) \leq \frac{1}{d(\hat v,\, \hat \ledir(B))\norm{B}^2} \leq e^{-\gamma}
    \end{align*}
    and
    \begin{align*}
        d(A^{-1}\, \hat w,\, \hat \ledir(A)) \leq \frac{1}{d(\hat w,\, \hat \medir^*(A))\norm{A}^2} \leq e^{-\gamma}.
    \end{align*}
    Using additionally Lemma \ref{L1},
    \begin{align*}
        d(\hat \medir^*(B),\, B\, \hat e_1)
        &\leq\frac{1}{d(\hat e_1,\, \hat \ledir(B))\norm{B}^2}
        \leq \frac{\norm{B}}{\norm{B}^2\, \sqrt{\norm{B\, e_1}^2 - \norm{B}^{-2}}}\\
        &\leq e^{-\gamma}\frac{1}{\sqrt{e^{2\gamma} - e^{-2\gamma}}} \leq 2\, e^{-2\gamma},
    \end{align*}
    and similarly,
    \begin{align*}
        d(\ledir(A),\, A^{-1}\, \hat e_2) \leq \, 2\, e^{-2\gamma},
    \end{align*}
    for every $\gamma\geq \gamma_0$, for some $\gamma_0$ sufficiently large.
    
    Thus, by triangular inequality, and the fact that $B\, \hat e_1 = A^{-1}\, \hat e_2$,
    \begin{align*}
        d(B\, \hat v,\, A^{-1}\, \hat w)
        &\leq d(B\, \hat v,\, \hat \medir^*(B)) + d(\hat \medir^*(B),\, B\, \hat e_1)\\
        &+ d(A^{-1}\, \hat e_2,\, \hat \ledir(A)) + d(\hat \ledir(A),\, A^{-1}\, \hat w)\\
        &\leq 2\,(e^{-\gamma} + 2e^{-2\gamma}) \leq 3\, e^{-\gamma},
    \end{align*}
    for every $\gamma\geq \gamma_0$.
\end{proof}

\subsection{Conditions to ensure Assumptions 1-3}
\label{subsection: Conditions to gurantee 1-3}

\begin{lemma}
\label{lemma: det(I+t E) formula}
    For any matrix $E\in \Mat_2(\R)$ and $t\in \R$,
    \begin{align*}
        det(I + t\, E) = 1 + t\, \tr (E) + t^2\, \det (E).
    \end{align*}
    In particular, if $E\in \GL_2(R)$,
    \begin{align*}
        \det(I+ t\, E)
        =(\det E)\,\left[ 
            \left(
                t+\frac{\tr E}{2\,\det E} \right
            )^2 + \frac{\Delta_E}{4\,(\det E)^2}
        \right],
    \end{align*}
    where $\Delta_E:= 4(\det E) - (\tr E)^2$.
\end{lemma}
\begin{proof}
    Direct computation.
\end{proof}

Consider matrices $A,\, E\in \Mat_2(\R)$. In the rest of this appendix we write $A_t := A\, (I + t\, E)$ and $Q_{A_t}(v):= (A_t v)\wedge (\dot A_t v)$.
  Define  $J:=\begin{bmatrix}
	0 & 1 \\ -1 & 0
\end{bmatrix}$ and  
$$ E^\sharp:= (J\,E  + (J\,E)^t)/2 =  (J\,E -E^t\, J)/2.$$

\begin{proposition} For all $v\in \R^2$,
$$ Q_{A_t}(v)=(\det A)\, ( v\wedge E\, v ) = 
(\det A)\, ( v^t\, E^\sharp \, v  ) .$$ 
\end{proposition}

\begin{proof} Since $(A E v) \wedge (A E v)=0$,
\begin{align*}
Q_{A_t}(v) &= (A_t v)\wedge (\dot A_t v) = (A (I+t E) v )\wedge  (A E v) \\
& =  (A  v )\wedge  (A E v) = (\det A)\, ( v \wedge  (E v)   )\\
& =   (\det A)\, \langle v,\, J E\,  v\rangle  =   (\det A)\, ( v^t \, E^\sharp\,  v)   .
\end{align*}
\end{proof}

Consider the function $\Delta:\Mat_2(\R)\to \R$  in Lemma~\ref{lemma: det(I+t E) formula}.
\begin{proposition}
	\label{pos winding charact}
The following are equivalent:
\begin{enumerate}
	\item The quadratic form $Q_{A_t}$ is positive (resp. negative) definite;
	\item  $\Delta_E>0$  and $e_{12}<0<e_{21}$ (resp. $\Delta_E>0$  and  $e_{21}<0<e_{12}$);
	\item $E$ has no real eigenvalues and the solutions of the linear O.D.E. \, $\dot X=E\, X$
	wind positively (resp. negatively) around the origin.
\end{enumerate}
\end{proposition}
\begin{proof}
 Simple calculations give
\begin{equation}
\label{E sharp characterization}
E^\sharp = 
\begin{bmatrix}
	e_{21} & \frac{e_{22} - e_{11}}{2} \\
	\frac{e_{22} - e_{11}}{2}  &  -e_{12}
\end{bmatrix}\; \text{ and } \; \det E^\sharp = \frac{1}{4}\, \Delta_E  
\end{equation} 
from which the equivalence (1)$\Leftrightarrow$ (2) follows.
For the equivalence (1)$\Leftrightarrow$ (3)
notice that $v\wedge (E v)=0$ if and only if
$v$ is an eigendirection of $E$ associated with some real eigenvalue. Therefore, the quadratic form $Q_{A_t}$
is definite (positive or negative) if and only if
$E$ has no real eigenvalues and this happens exactly when the discriminant  $-\Delta_E$ of the polynomial $\det(I+\lambda\, E)$ is strictly negative.
Finally notice that since the curves $A\, e^{t E}\,v$ and $A\, (I+t\, E)\, v$ are tangent at $t=0$, they wind in the same direction.
\end{proof}

\begin{proposition}
	\label{pos winding charact, case Delta=0}
	The following are equivalent:
	\begin{enumerate}
		\item The quadratic form $Q_{A_t}$ is positive (resp. negative) semi-definite but not definite;
		\item  $\Delta_E=0$, $e_{12}\leq 0\leq e_{21}$ and $e_{12}<e_{21}$  (resp. $\Delta_E=0$,  $e_{21}\leq 0\leq e_{12}$ and $e_{12}>e_{21}$);
		\item $E$ has a double real eigenvalue and the solutions of the linear O.D.E. \, $\dot X=E\, X$
		wind positively (resp. negatively) around the origin.
	\end{enumerate}
\end{proposition}

\begin{proof}
Same argument as in the proof of Proposition~\ref{pos winding charact}.
\end{proof}

Consider the seminorm $\Xi$ and the  following sets
introduced in Subsection~\ref{subsection: conditions for winding}. 
\begin{align*}
	\Gamma_+ &:=\left\{ E \in\Mat_2(\R) \colon \Xi(E)>0,\, \Delta_E\geq 0  \, \text{ and }\,  e_{21}\leq 0\leq e_{12} \right\} \\
	\Gamma_- &:=\left\{ E \in\Mat_2(\R) \colon \Xi(E)>0,\, \Delta_E\geq 0 \, \text{ and }\,  e_{12}\leq 0\leq e_{21} \right\} 
\end{align*}
The intersection of these sets  is empty, i.e., $\Gamma_-\cap \Gamma_+=\emptyset$.

\begin{proposition}
\label{proposition: Gamma characterization}
Using the previous definitions:
\begin{align*}
	\Gamma_+ &=\{ E\in\Mat_2(\R)\, \colon \, Q_{I+t\,E} \, \text{ is positive definite or semi-definite } \} \\
	\Gamma_- &=\{ E\in\Mat_2(\R)\, \colon \, Q_{I+t\,E} \, \text{ is negative definite or  semi-definite } \} 
\end{align*}
\end{proposition}

\begin{proof}
Notice that
$$\Delta_E=4\,(e_{11}\, e_{22}- e_{21}\, e_{21}) -(e_{11}+e_{22})^2  = -4\, e_{12}\, e_{21}-(e_{11}-e_{22})^2  $$
so that if \, $\Delta_E=0$ and $e_{12}\leq 0 \leq e_{21}$ then  \;    
$$E\in\Gamma_+ \; \Leftrightarrow\;  \Xi(E)>0\; \Leftrightarrow\;  e_{12}< e_{21}\; \Leftrightarrow\;  Q_{I+t\,E} \, \text{ is positive  semi-def. } $$
Similarly, if \, $\Delta_E>0$ and $e_{12}\leq 0 \leq e_{21}$ then \,    
$$E\in\Gamma_+ \; \Leftrightarrow\;   \Xi(E)>0\; \Leftrightarrow\;  e_{12}< e_{21} \; \Leftrightarrow\;  Q_{I+t\,E} \, \text{ is positive  definite. }$$
An entirely analogous argument  works  for $\Gamma_-$.
\end{proof}

\begin{example}
	
In the case of Schr\"odinger matrices,
\begin{align*}
    S(x - t) :=
    \begin{bmatrix}
        x-t & -1\\
        1 & 0
    \end{bmatrix},
\end{align*}
we can write
\begin{align*}
    S(x - t) =
    \begin{bmatrix}
        x & -1\\
        1 & 0
    \end{bmatrix}
    \left(
        I + t\, 
        \begin{bmatrix}
            0 & 0\\
            1 & 0
        \end{bmatrix}
    \right) =: A\, (I + t\, E).
\end{align*}
with $E \in \Gamma_+$.
\end{example}

\medskip

\subsection*{Acknowledgments}
J.B., A.C., P.D. and C.F. were  supported by FCT-Funda\c{c}\~{a}o para a Ci\^{e}ncia e a Tecnologia through the project  PTDC/MAT-PUR/29126/2017.

J.B.  was also  supported by  the Center of Excellence "Dynamics, Mathematical Analysis and Artificial Intelligence" at Nicolaus Copernicus University in Torun.

P.D.   was also supported by CMAFCIO through FCT project  UIDB/04561/2020.

A.C. was also supported by a FAPERJ postdoctoral grant.

S.K. was supported by the CNPq research grant 313777/2020-9 and  by the Coordena\c{c}\~ao de Aperfei\c{c}oamento de Pessoal de N\'ivel Superior - Brasil (CAPES) - Finance Code 001.

\bigskip

%\nocite{*}
\bibliographystyle{amsplain}
%\bibliography{bib}

\providecommand{\bysame}{\leavevmode\hbox to3em{\hrulefill}\thinspace}
\providecommand{\MR}{\relax\ifhmode\unskip\space\fi MR }
\providecommand{\MRhref}[2]{%
	\href{http://www.ams.org/mathscinet-getitem?mr=#1}{#2}
}
\providecommand{\href}[2]{#2}

\end{document}